\documentclass[10pt,a4paper]{amsart}

\usepackage[utf8]{inputenc}

\usepackage{amsmath}
\usepackage{amsfonts}
\usepackage{amssymb}
\usepackage{latexsym}

\usepackage{graphicx} 
\usepackage{caption} 
\usepackage{subcaption} 
\usepackage{enumerate} 

\usepackage{multirow} 
\usepackage{longtable} 
\usepackage{array} 

\usepackage{tikz} 
\usetikzlibrary{snakes,arrows,decorations.pathmorphing,decorations.pathreplacing}
\tikzset{vertex/.style={circle,fill=black,inner sep=1pt,outer sep=2pt},
         nvertex/.style={circle,fill=black,inner sep=0.7pt,outer sep=2pt},
         yvertex/.style={font=\small,minimum size=6pt},         
         kvertex/.style={font=\scriptsize,minimum size=6pt},
         xvertex/.style={font=\small,minimum size=8pt},
         bvertex/.style={font=\small, inner sep=1pt},
         mvertex/.style={rectangle,draw=black,thick,inner sep=2pt,outer sep=2pt},
         tvertex/.style={inner sep=1pt, font=\small},
         unvertex/.style={circle,fill=white,draw=white,inner sep=1pt},
         fill1/.style={fill=black!20,draw=black!20},
         fill2/.style={fill=black!40,draw=black!40},
         fill12/.style={fill=black!60,draw=black!60},
         >=stealth',
         leadsto/.style={-angle 90,decorate,decoration=snake,very thick},
         cut/.style={decorate,decoration=saw,very thick}}
\newcommand{\replacevertex}[3][fill=white,draw=white]
 {  \node at #2 [#1,circle,inner sep=1pt] {};
  \node #2 at #2 #3;
 }

\usepackage[all]{xy} 

\usepackage{textcomp} 

\newtheorem*{theorem*}{Theorem}
\newtheorem{theorem}{Theorem}
\theoremstyle{plain}
\newtheorem{corollary}[theorem]{Corollary}
\newtheorem{definition}[theorem]{Definition}
\newtheorem{lemma}[theorem]{Lemma}
\newtheorem{proposition}[theorem]{Proposition}

\theoremstyle{definition}
\newtheorem{exmp}{Example}[section]

\renewcommand{\mod}{\operatorname{mod}\nolimits}

\newcommand{\ant}{\ensuremath{\widetilde{A}_n}}

\newcommand{\smod}{\ensuremath{\underline{\mathrm{mod}}~\hat{\Lambda}}}
\newcommand{\kbp}[1][\Lambda]{\ensuremath{\mathcal{K}^b(\operatorname{proj} #1)}}
\newcommand{\zainfinf}{\ensuremath{\mathbb{Z}A_{\infty}^{\infty}}}
\newcommand{\zainf}{\ensuremath{\mathbb{Z}A_{\infty}}}

\begin{document}
\title[AR-comp. of $\kbp$ for a cl.t. alg of type $\widetilde{A}$]{The
  Auslander-Reiten components of $\kbp$ for a
  cluster-tilted algebra of type $\widetilde{A}$}
\author{Kristin Krogh Arnesen}
\author{Yvonne Grimeland}

\address{Institutt for matematiske fag\\ NTNU\\ 7491 Trondheim\\ Norway}
\email{Kristin.Arnesen@math.ntnu.no}
\email{yvonne\_grimeland@hotmail.com}

\begin{abstract}
We classify the Auslander-Reiten components of $\kbp$, where $\Lambda$ is
a cluster tilted algebra of type $\widetilde{A}$. The main tool is the
combinatoric description of the indecomposable complexes in $\kbp$ via
homotopy strings and homotopy bands.
\end{abstract}

\maketitle

\section{Introduction}
The derived category of an abelian category was introduced by
Grothen\-dieck and Verdier in the early 1960s, and in the 1980s,
Happel started studying the derived category of a finite dimensional
algebra \cite{happel}.  The module category is embedded into the
derived category of the algebra, and this expansion to larger
categories provided a new tool for comparing and distinguishing the
module categories of two algebras.

Let $\Lambda$ be a finite dimensional $k$-algebra, where $k$ is an
algebraically closed field, and let $\mathrm{mod}~ \Lambda$ denote the
category of finitely generated left $\Lambda$-modules. The derived
category of $\Lambda$ is denoted by $D(\mod \Lambda)$, with suspension
functor called shift and denoted by $[1]$. Two important subcategories
is the bounded derived category, denoted by $D^b(\mod \Lambda)$, and
its subcategory $\kbp$, the bounded homotopy category of finitely
generated projective $\Lambda$-modules.  One way of describing $D(\mod
\Lambda)$ is to describe its Auslander-Reiten (hereafter abbreviated
AR) structure. In general, the subcategory $\kbp$ has AR-triangles
whenever $\Lambda$ has finite Gorenstein dimension.  In particular,
there is an explicit description of the AR-structure when $\Lambda$ is
gentle \cite{bob}.

The gentle algebras form a subclass of the special biserial
algebras, introduced by Skowro\'nski and Waschb\"{u}sch in 1983
\cite{specialbiserial}. The AR-structure of the module category of a
gentle algebra $\Lambda\cong kQ/I$, where $I$ is an admissible ideal,
can be combinatorially described in terms of $Q$ and $I$
\cite{tamebiserial}. More recently, the AR-structure of $\kbp$ has
also been given combinatorially for gentle algebras.  This was done in
2011, when Bobi\'nski gave a combinatorial algorithm for computing the
AR-triangle starting in any given indecomposable object of $\kbp$,
where $\Lambda$ is gentle \cite{bob}.

The foundation for Bobi\'nski's algorithm is diverse.  In 2003,
Bekkert and Merklen showed that the indecomposable objects of $\kbp$
are the complexes arising from so called homotopy strings and
homotopy bands (in Bobi\'nski's renaming) \cite{bekkert}.  Moreover,
Bobi\'nski takes advantage of the Happel functor, introduced by
Happel in \cite{happel}.  The Happel functor is an exact functor of
triangulated categories $\Psi \! : D^b(\mod \Lambda) \rightarrow
\underline{\mathrm{mod}}~\hat{\Lambda}$, where the latter category is
the stable module category of the repetitive algebra of
$\Lambda$.  When $\Lambda$ is gentle, the repetitive algebra
$\hat{\Lambda}$ is combinatorially described in terms of strings, see
Ringel \cite{ringel}.  Since the repetitive algebra of a gentle
algebra is special biserial, it is even possible to describe the
AR-sequences of $\smod$, using the methods of Wald and Waschb\"usch
\cite{tamebiserial}. 

In its original appearance, the Happel functor is rather abstract and
no explicit construction is given.  Nevertheless, some of its
properties are quite remarkable: It is always full and faithful, and
if $\Lambda$ is of finite global dimension, it is a
triangle-equivalence.  It also extends the inclusion functor embedding
$\mathrm{mod}~\Lambda$ into $\smod$.  By using all the known structure
of $\Lambda$, $\mathrm{mod}~\Lambda$, $\smod$ and $\kbp$ when
$\Lambda$ is gentle, Bobi\'nski constructed a formula for the Happel
functor in the gentle case.

Special classes of gentle algebras include some \emph{cluster-tilted
  algebras}.  The cluster-tilted algebras were introduced in 2007 by
Buan, Marsh and Reiten \cite{BMR}.  In 2010, Assem, Br\"ustle,
Charbonneau-Jodoin and Plamondon showed that cluster-tilted algebras
of type $A$ and $\widetilde{A}$ are gentle \cite{abcp}.  The mutation
classes of type $\widetilde{A}$ and the derived equivalences between
cluster-tilted algebras of type $\widetilde{A}$ are described by
Bastian \cite{bastian}.  

Also worth noting is the derived invariant described by
Avella-Alaminos and Geiss \cite{diana}. Using this invariant one can
find an upper bound for the number of AR-components containing
sequences with only one middle term.

In this paper, we classify the AR-components of $\kbp$, where
$\Lambda$ is a cluster-tilted algebra of type $\widetilde{A}$. Our
main tool is Bobi\'nski's algorithm for computing AR-triangles in
$\kbp$.  

Parallel to this paper, a paper in progress by Fedra Babaei describes
the AR-structure of $\kbp$ where $\Lambda$ is a
cluster-tilted algebra of type $A$.

The paper is organized as follows: In Section 2, we state the main
result of the paper. Section 3 is an overview of the background theory
needed.  In Section 4, we give the details needed from Bastian's
description of the mutation classes of type $\widetilde{A}$. In
Section 5, we introduce the combinatorial concepts of walks and
reductions, which we use to restate Bobi\'nski's algorithm for our
class of algebras in Section 6.  The main result is proved in Section
7.  Finally, in Section 8, we give an example.  Some technical
results are proved in the appendices.

We would like to thank the referee for very helpful comments.
\section{Main results} \label{sec:main} 

In this section, we state the main result. By a quiver, we mean a pair
$Q = (Q_0, Q_1)$ where $Q_0$ is a set of vertices and $Q_1$ is a set
of arrows, together with functions $s,t: Q_1 \rightarrow Q_0$
returning the starting and ending vertex of an arrow, respectively.

Let $Q$ be the fixed quiver given by the parameters $x$, $y$, $x'$ and
$y'$, as shown in Figure \ref{firstquiver}.  We require that at least
one of $x, y$ and at least one of $x',y'$ are non-zero.

We then define the double quiver $Q'$ to be the quiver with vertices
$Q'_0 = Q_0$ and arrows $Q'_1 = Q_1 \cup Q_1^{-1}$. Let a
\emph{homotopy string} denote a path in $Q'$ with no subpath of the
form $\alpha\alpha^{-1}$ or $\alpha^{-1}\alpha$ for any $\alpha \in
Q_1$.  By a \emph{central homotopy string}, we mean a homotopy string
both starting and ending in a vertex marked with $\vartriangle$ or
$\blacktriangle$, excluding the homotopy strings starting with the
arrows $\alpha_1$ or $\beta_1$ and the homotopy strings ending with
the arrows $\alpha_1^{-1}$ or $\beta_1^{-1}$, and the trivial homotopy
strings for the vertices marked with $\blacktriangle$. (Note that in
case one or more of the parameters are zero, then any vertex which is
adjacent to both an oriented cycle of length $3$ and an arrow
$\alpha_i$ or $\beta_j$ should be marked with $\blacktriangle$, and
that any vertex which is adjacent to two oriented cycles of length $3$
should be marked with $\vartriangle$.) A \emph{homotopy band} is a
central homotopy string starting and ending in the same vertex, with
some additional constraints (see Section \ref{subhomstring}).

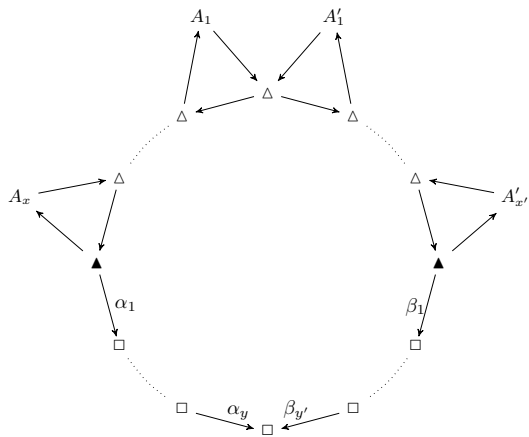
\begin{figure}[h!]
\begin{minipage}{\textwidth}
\[ \scalebox{.75}{\begin{tikzpicture}
  \node (1) at (30:3cm) [yvertex] {$\vartriangle$};
  \node (2) at (60:3cm) [yvertex] {$\vartriangle$};
  \node (3) at (90:3cm) [yvertex] {$\vartriangle$};
  \node (4) at (120:3cm) [yvertex] {$\vartriangle$};
  \node (5) at (150:3cm) [yvertex] {$\vartriangle$};
  \node (6) at (180:3cm) [yvertex] {$\blacktriangle$};
  \node (7) at (210:3cm) [yvertex] {$\Box$};
  \node (8) at (240:3cm) [yvertex] {$\Box$};
  \node (9) at (270:3cm) [yvertex] {$\Box$};
  \node (10) at (300:3cm) [yvertex] {$\Box$};
  \node (11) at (330:3cm) [yvertex] {$\Box$};
  \node (12) at (360:3cm) [yvertex] {$\blacktriangle$};
  \node (t1) at (105:4.5cm) [yvertex] {$A_1$};
  \node (t2) at (165:4.5cm) [yvertex] {$A_{x}$};
  \node (t3) at (75:4.5cm) [yvertex] {$A'_1$};
  \node (t4) at (15:4.5cm) [yvertex] {$A'_{x'}$};
  \draw [->] (3)--(2);
  \draw [->] (3)--(4);
  \draw [->] (5)--(6);
  \draw [->] (6)--(7);
  \draw [->] (8)--(9);
  \draw [->] (10)--(9);
  \draw [->] (12)--(11);
  \draw [->] (1)--(12);
  \draw [dotted] (4) .. controls (135:3cm) .. (5);
  \draw [dotted] (7) .. controls (225:3cm) .. (8);
  \draw [dotted] (11) .. controls (315:3cm) .. (10);
  \draw [dotted] (2) .. controls (45:3cm) .. (1);
  \draw [->] (t1)--(3);
  \draw [->] (4)--(t1);
  \draw [->] (t2)--(5);
  \draw [->] (6)--(t2);
  \draw [->] (t3)--(3);
  \draw [->] (2)--(t3);
  \draw [->] (t4)--(1);
  \draw [->] (12)--(t4);
  \node [right] at (195:2.9cm) {$\alpha_{1}$};
  \node [above] at (260:2.95cm) { $\alpha_{y}$};
  \node [above] at (280:2.95cm) { $\beta_{y'}$};
  \node [left] at (345:3cm) { $\beta_{1}$};
\end{tikzpicture} }
\]
\end{minipage}
\caption{The quiver $Q$ given by $x$, $y$, $x'$ and $y'$.}\label{firstquiver}
\end{figure}

Let $\Lambda$ be the algebra $kQ/I$, where $I$ is the ideal generated
by all compositions of two arrows in each directed cycle of length $3$
in $Q$, and $kQ$ is the path algebra of $Q$.  This is in fact a
cluster-tilted algebra of type $\ant$, and all cluster-tilted algebras
of type $\ant$ are of this form, up to derived equivalence
\cite{bastian}. We will now describe the AR-components of $\kbp$.
First we state their types and numbers in the following theorem.  Let
$\tau$ denote the AR-translate in $\kbp$.

\begin{theorem}\label{mainresult1}
  Let $Q$ be a quiver as in Figure \ref{firstquiver}, and let $\Lambda
  = kQ/I$ where $I$ is as described above. Then the AR-quiver of
  $\kbp$ consists of:
\begin{enumerate}[(i)]
\item A class of tubes of rank one (homogeneous tubes), where up to
  shift, the tubes are parametrized by the set of pairs consisting of one
  homotopy band and one element of $k$.
\item A class of components given by the parameters $x$ and $y$.  If
  $x = 0$, we get up to shift a tube of rank $y$.  If $x >
  0$, we get $x$ components of type $\zainf$ with $\tau^{x+y} = [x]$.
\item A class of components given by the parameters $x'$ and $y'$.
  If $x' = 0$, we get up to shift a tube of rank $y'$.  If $x' > 0$,
  we get $x'$ components of type $\zainf$ with $\tau^{x'+y'} = [x']$.
  \item Up to shift, one $\zainfinf$-component containing all the stalk
    complexes corresponding to vertices marked by $\Box$ and
    $\blacktriangle$. 
  \item Up to shift, a class of $\zainfinf$-components, parametrized by the
    central homotopy strings.
\end{enumerate}
\end{theorem}

For any quiver as in Figure \ref{firstquiver} the edge of the
components in (ii) can be described easily in terms of the
quiver. Figure \ref{edge} shows the edge of a $\zainf$-component
where $\tau^{x+y} = [x]$.  The edge of a $\zainf$-component
where $\tau^{x'+y'} = [x']$ can be found symmetrically.

\begin{figure}[h]
\[
\resizebox{\linewidth}{!}{
\begin{tikzpicture}[scale=0.9]
  \node (0) at (-0.2,0.5) [kvertex] {$\cdots$};
  \node (1) at (0,0) [kvertex] {$P_{A_x}$};
  \node (s1) at (0,-0.4) [kvertex] {$[i]$};
  \node (2) at (2,0) [kvertex] {$P_{A_{x-1}}$};
  \node (s2) at (2,-0.4) [kvertex] {$[i-1]$};
  \node (3) at (3.35,0) [kvertex] {$\cdots$};
  \node (4) at (4.7,0) [kvertex] {$P_{A_1}$};
  \node (s4) at (4.7,-0.4) [kvertex] {$[i-x+1]$};
  \node (5) at (6.8,0) [kvertex] {$(P_{t\alpha_y} \rightarrow
    P_{s\alpha_y})$};
  \node (s5) at (6.8,-0.4) [kvertex] {$[i-x+1]$};
  \node (7) at (8.25,0) [kvertex] {$\cdots$};
  \node (6) at (9.7,0) [kvertex] {$(P_{t\alpha_1} \rightarrow
    P_{s\alpha_1})$};
  \node (s6) at (9.7,-0.4) [kvertex] {$[i-x+1]$};
  \node (8) at (11.7,0) [kvertex] {$P_{A_x}$};
  \node (s8) at (11.7,-0.4) [kvertex] {$[i-x]$};
  \node (9) at (3.35,1) [kvertex] {$\cdots$};
  \node (10) at (8.25,1) [kvertex] {$\cdots$};
  \node (11) at (11.9, 0.5) [kvertex] {$\cdots$};
  \draw [->] (1)--(0.9,1);
  \draw [->] (1.1,1)--(2);
  \draw [->] (2) -- (2.9,1);
  \draw [->] (3.8,1)--(4);
  \draw [->] (4)--(5.6,1);
  \draw [->] (5.8,1)--(5);
  \draw [->] (5)--(7.7,1);
  \draw [->] (8.8,1)--(6);
  \draw [->] (6)--(10.6,1);
  \draw [->] (10.8,1)--(8);
\end{tikzpicture}}
\]
\caption{The edge of an AR-component of type $\zainf$, with the degree shown
  below each complex.}\label{edge}
\end{figure}
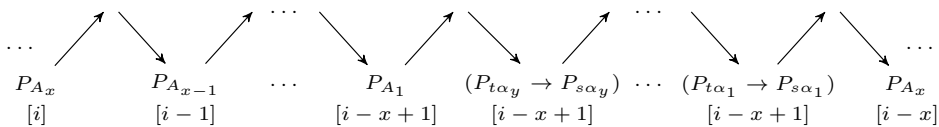

\begin{exmp}\label{ex:hovedeks}
  We now consider the quiver $Q$ given in Figure \ref{firstex}, and
  the path algebra $\Lambda = kQ/I$ where $I =\left\langle ih, gi, hg,
    ed, fe, df, ba, cb, ac, ts, ut, su, qp, rq, pr \right\rangle$.
  Figures \ref{edge1} and \ref{edge2} show the edges of two
  AR-components of type $\zainf$, one given by $x$ and $y$, and one
  given by $x'$ and $y'$.  The first component has the property
  $\tau^5 = [3]$ and the second component has the property $\tau^6 =
  [2]$.

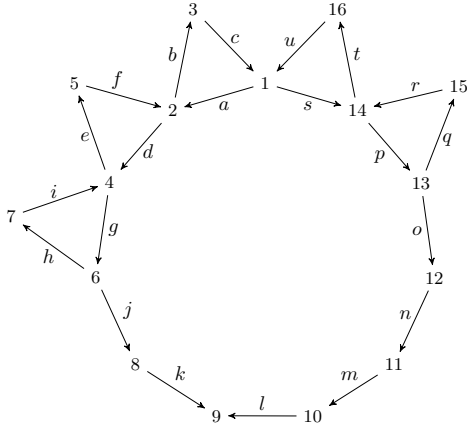
\begin{figure}[p]
\begin{minipage}{\textwidth}
\[\scalebox{.75}{\begin{tikzpicture}
  \node (1) at (90:3cm) [yvertex] {$1$};
  \node (2) at (122.7:3cm) [yvertex] {$2$};
  \node (4) at (155.4:3cm) [yvertex] {$4$};
  \node (6) at (188.1:3cm) [yvertex] {$6$};
  \node (8) at (220.8:3cm) [yvertex] {$8$};
  \node (9) at (253.5:3cm) [yvertex] {$9$};
  \node (10) at (286.2:3cm) [yvertex] {$10$};
  \node (11) at (318.9:3cm) [yvertex] {$11$};
  \node (12) at (351.6:3cm) [yvertex] {$12$};
  \node (13) at (24.3:3cm) [yvertex] {$13$};
   \node (14) at (57.3:3cm) [yvertex] {$14$};
  
  \node (3) at (106.3:4.5cm) [yvertex] {$3$};
  \node (5) at (138:4.5cm) [yvertex] {$5$};
  \node (7) at (171.7:4.5cm) [yvertex] {$7$};
  \node (16) at (73.6:4.5cm) [yvertex] {$16$};
  \node (15) at (41:4.5cm) [yvertex] {$15$};
  \draw [->] (1)--(2);
  \draw [->] (2)--(3);
  \draw [->] (3)--(1);
  \draw [->] (2)--(4);
  \draw [->] (4)--(5);
  \draw [->] (5)--(2);
  \draw [->] (4)--(6);
  \draw [->] (6)--(7);
  \draw [->] (7)--(4);
  \draw [->] (6)--(8);
  \draw [->] (8)--(9);
  \draw [->] (10)--(9);
  \draw [->] (11)--(10);
  \draw [->] (12)--(11);          
  \draw [->] (13)--(12);
  \draw [->] (13)--(15);
  \draw [->] (15)--(14);
  \draw [->] (14)--(13);
  \draw [->] (14)--(16);
  \draw [->] (16)--(1);
  \draw [->] (1)--(14);
  \node [below] at (104.3:2.9cm) {$a$};
  \node [below] at (135:2.9cm) {$d$};
  \node [right] at (171.7:2.9cm) {$g$};
  \node [right] at (201.4:2.8cm) {$j$};
  \node [above] at (238:2.8cm) {$k$};
  \node [above] at (269:2.9cm) {$l$};
  \node [above] at (302:2.8cm) {$m$};
  \node [left] at (338:2.9cm) {$n$};
  \node [left] at (8:2.9cm) {$o$};
  \node [below] at (44:2.8cm) {$p$};
  \node [below] at (75:2.9cm) {$s$};
  \node [left] at (95:3.8cm) {$c$};
  \node [left] at (112:3.8cm) {$b$};
  \node [above] at (133:3.8cm) {$f$};
  \node [left] at (145:3.6cm) {$e$};
  \node [above] at (167:3.8cm) {$i$};
  \node [below] at (177:3.8cm) {$h$};
  \node [right] at (34:3.6cm) {$q$};
  \node [above] at (46:3.8cm) {$r$};
  \node [right] at (68:3.8cm) {$t$};
  \node [left] at (80:3.8cm) {$u$};
\end{tikzpicture}}
\]
\end{minipage}
\caption{The quiver from Example \ref{ex:hovedeks}.}\label{firstex}
\end{figure}

\begin{figure}[p]
\[
\begin{tikzpicture}
  \node (0) at (-0.5,0.5) [kvertex] {$\cdots$};
  \node (1) at (0,0) [kvertex] {$P_{7}$};
  \node (s1) at (0,-0.4) [kvertex] {$[0]$};
  \node (2) at (2,0) [kvertex] {$P_{5}$};
  \node (s2) at (2,-0.4) [kvertex] {$[-1]$};
  \node (4) at (4,0) [kvertex] {$P_{3}$};
  \node (s4) at (4,-0.4) [kvertex] {$[-2]$};
  \node (5) at (6,0) [kvertex] {$(P_{9} \rightarrow
    P_{8})$};
  \node (s5) at (6,-0.4) [kvertex] {$[-2]$};
  \node (6) at (8,0) [kvertex] {$(P_{8} \rightarrow
    P_{6})$};
  \node (s6) at (8,-0.4) [kvertex] {$[-2]$};
  \node (8) at (10,0) [kvertex] {$P_{7}$};
  \node (s8) at (10,-0.4) [kvertex] {$[-3]$};
  \node (11) at (10.5, 0.5) [kvertex] {$\cdots$};
  \draw [->] (1)--(0.9,1);
  \draw [->] (1.1,1)--(2);
  \draw [->] (2) -- (2.9,1);
  \draw [->] (3.1,1)--(4);
  \draw [->] (4)--(4.9,1);
  \draw [->] (5.1,1)--(5);
  \draw [->] (5)--(6.9,1);
  \draw [->] (7.1,1)--(6);
  \draw [->] (6)--(8.9,1);
  \draw [->] (9.1,1)--(8);
\end{tikzpicture}
\]
\caption{The edge of a $\zainf$-component where $\tau^5 = [3]$.}\label{edge1}
\end{figure}

\begin{figure}[p]
\[
\resizebox{\linewidth}{!}{
\begin{tikzpicture}[scale=0.95]
  \node (0) at (-0.5,0.5) [kvertex] {$\cdots$};
  \node (1) at (0,0) [kvertex] {$P_{15}$};
  \node (s1) at (0,-0.4) [kvertex] {$[0]$};
  \node (2) at (2,0) [kvertex] {$P_{16}$};
  \node (s2) at (2,-0.4) [kvertex] {$[-1]$};
  \node (4) at (4,0) [kvertex] {$(P_{9} \rightarrow
    P_{10})$};
  \node (s4) at (4,-0.4) [kvertex] {$[-1]$};
  \node (5) at (6,0) [kvertex] {$(P_{10} \rightarrow
    P_{11})$};
  \node (s5) at (6,-0.4) [kvertex] {$[-1]$};
  \node (6) at (8,0) [kvertex] {$(P_{11} \rightarrow
    P_{12})$};
  \node (s6) at (8,-0.4) [kvertex] {$[-1]$};
  \node (7) at (10,0) [kvertex] {$(P_{12} \rightarrow
    P_{13})$};
  \node (s7) at (10,-0.4) [kvertex] {$[-1]$};
  \node (8) at (12,0) [kvertex] {$P_{15}$};
  \node (s8) at (12,-0.4) [kvertex] {$[-2]$};
  \node (11) at (12.5, 0.5) [kvertex] {$\cdots$};
  \draw [->] (1)--(0.9,1);
  \draw [->] (1.1,1)--(2);
  \draw [->] (2) -- (2.9,1);
  \draw [->] (3.1,1)--(4);
  \draw [->] (4)--(4.9,1);
  \draw [->] (5.1,1)--(5);
  \draw [->] (5)--(6.9,1);
  \draw [->] (7.1,1)--(6);
  \draw [->] (6)--(8.9,1);
  \draw [->] (9.1,1)--(7);
  \draw [->] (7)--(10.9,1);
  \draw [->] (11.1,1)--(8);
\end{tikzpicture}}
\]
\caption{The edge of a $\zainf$-component where $\tau^6 = [2]$.}\label{edge2}
\end{figure}

\begin{figure}[p]
\[ \scalebox{0.8}{ \begin{tikzpicture}[scale=.75,yscale=-1]
 \foreach \x in {1,...,7}
  \foreach \y in {1,3,5,7}
   \node (\y-\x) at (\x*2,\y) [vertex] {};
 \foreach \x in {1,...,6}
  \foreach \y in {0,2,4,6,8}
   \node (\y-\x) at (\x*2+1,\y) [vertex] {};
 \replacevertex{(7-4)}{ [tvertex] {$P_{6}$}}
 \replacevertex{(6-3)}{[yvertex] {$P_{8}$}} 
 \replacevertex{(5-3)}{[yvertex] {$P_{9}$}} 
 \replacevertex{(4-3)}{[yvertex] {$P_{10}$}}
 \replacevertex{(3-4)}{[yvertex] {$P_{11}$}}
 \replacevertex{(2-4)}{[yvertex] {$P_{12}$}}
 \replacevertex{(1-5)}{[yvertex] {$P_{13}$}}   
 \foreach \xa/\xb in {1/2,2/3,3/4,4/5,5/6,6/7}
  \foreach \ya/\yb in {1/2,3/4,5/4,7/6}
   {
    \draw [->] (\ya-\xa) -- (\yb-\xa);
    \draw [->] (\yb-\xa) -- (\ya-\xb);
   }
 \foreach \xa/\xb in {1/2,2/3,3/4,4/5,5/6,6/7}
  \foreach \ya/\yb in {1/0,3/2,5/6,7/8}
   {
    \draw [->] (\ya-\xa) -- (\yb-\xa);
    \draw [->] (\yb-\xa) -- (\ya-\xb);
   }
\node (1) at (1,4) [yvertex]{$\cdots$};
\node (2) at (15,4) [yvertex]{$\cdots$};
\node (3) at (8,-0.5) [yvertex]{$\vdots$};
\node (4) at (8,8.5) [yvertex]{$\vdots$};
\end{tikzpicture} } \]
\caption{The special $\zainfinf$-component.}\label{uglycomp}
\end{figure}
\end{exmp}

Figures \ref{edge1} and \ref{edge2} show the edges of two
AR-components of type $\zainf$, one given by $x$ and $y$, and one
given by $x'$ and $y'$.  The first component has the property $\tau^5
= [3]$ and the second component has the property $\tau^6 = [2]$.

Figure \ref{uglycomp} shows part of the structure in the special
$\zainfinf$-component.  In particular, it shows the irreducible maps
between the stalk complexes.

The remaining $\zainfinf$-components are parametrized by the central
homotopy strings, that is, each component contains exactly one central
homotopy string. Note that also the trivial homotopy strings
corresponding to the vertices $4$, $2$, $1$ and $14$ are central
homotopy strings.

\section{Background}
In this section we give an overview of the theory needed to prove
Theorem \ref{mainresult1}. First we give the definition of a gentle
algebra.  Then we describe $\kbp$ via homotopy strings and homotopy
bands, for the gentle algebra $\Lambda$.  Finally, we state some
results about the almost split triangles and components in the
AR-quiver of $\kbp$.

\subsection{Gentle algebras}\label{gentledef}
Let $\Lambda$ be isomorphic to $kQ/I$, for some quiver $Q$ and some
admissible ideal $I$.  Then $\Lambda$ is called \textit{special
  biserial} ~\cite{specialbiserial} if the following are satisfied:
\begin{enumerate}[(a)]
  \item for each vertex $x$ of $Q$ there are at most two arrows
    starting in $x$ and at most two arrows ending in $x$, and
  \item for any arrow $\alpha$ in $Q$ there is at most one arrow
    $\beta$ such that $\alpha\beta \notin I$ and at most
    one arrow $\gamma$ such that $\gamma \alpha \notin I$.
\end{enumerate}
Furthermore, if $I$ consists of only zero-relations, then $\Lambda$ is
a \textit{string algebra}, and if in addition all the relations in $I$
have length $2$ and for any arrow $\alpha$ in $Q$ there is at most
one arrow $\beta$ such that $\alpha\beta \in I$ and at most one arrow
$\gamma$ such that $\gamma \alpha \in I$, then $\Lambda$ is a
\textit{gentle} algebra.

We now state an equivalent definition of a gentle algebra, see
~\cite{bob}. This definition will be used later in the paper.
\begin{definition}\label{gentlealternativ}
  A finite dimensional algebra $\Lambda = kQ/I$ is \emph{gentle} if
  there exist two functions $S,T: Q_1 \rightarrow \{-1,1\}$ satisfying
  the following:
  \begin{enumerate}[(a)]
    \item if $\alpha \neq \beta$ start in the same vertex, then
      $S\alpha = -S\beta$,
    \item if $\alpha \neq \beta$ end in the same vertex, then $T\alpha
      = -T\beta$,
    \item if $\alpha$ starts in the vertex where $\beta$ ends, and
      $\alpha \beta$ is not in $I$, then $S\alpha = -T\beta$,
    \item if $\alpha$ starts in the vertex where $\beta$ ends, and
      $\alpha \beta$ is in $I$, then $S\alpha = T\beta$.
  \end{enumerate}
\end{definition}

\subsection{Homotopy strings and the category
  $\kbp$}\label{subhomstring}
In this subsection we first introduce the concept of homotopy strings
for a gentle algebra $\Lambda$, and given a homotopy string explain
how one can construct an associated string complex in $\kbp$. We also
discuss some special homotopy strings called homotopy bands, which in
addition to string complexes give rise to band complexes in $\kbp$.
Finally, we state a result giving the connection between
indecomposable objects of $\kbp$ and homotopy strings and bands.

Let $\Lambda\cong kQ/I$ be a gentle algebra. We now want to explain
what the homotopy strings associated with $\Lambda$ are. First we
define the double quiver $Q'$ of $Q$: Let $Q_0' = Q_0$, and $Q_1' =
Q_1 \cup Q_1^{-1}$, where $Q_1^{-1}$ is the set of formal inverses of
the arrows of $Q$; that is, for each $\alpha: x \rightarrow y$ in $Q$,
we have $\alpha^{-1}: y \rightarrow x$ in $Q'$. We also add formal
inverses of the trivial paths of $Q$: For a trivial path $1_x$ in $Q$
we add the inverse $1_x^{-1}$ to $Q'$.  

We define $(\alpha^{-1})^{-1} = \alpha$ and $(1_x^{-1})^{-1} = 1_x$
and extend the functions $s,t:Q'_1 \rightarrow Q_0$ to include
$s(1_x^{\varepsilon}) = t(1_x^{\varepsilon}) = x$ for all vertices $x$
and $\varepsilon \in \{-1,1\}$. We define the \emph{homotopy strings
  associated with $\Lambda$} to be all paths in $Q'$ which contain no
subpath of the form $\alpha\alpha^{-1}$ or $\alpha^{-1}\alpha$ for
$\alpha\in Q_1$. Note that each vertex $x$ in $Q_0$ gives rise to two
\emph{trivial homotopy strings}, namely the paths $1_x$ and $1_x^{-1}$
in $Q'$.  We also consider the empty homotopy string, denoted by
$\emptyset$.
	
A non-trivial, non-empty homotopy string $\omega$ can be written as
$\omega = \alpha_l \alpha_{l-1} \cdots \alpha_1$ where for each $1\leq
i \leq l$ the \emph{i}'th \emph{letter} $\alpha_i(\omega)=\alpha_i$ is
one arrow or the inverse of one arrow, and $l(\omega)=l$ is the number
of letters, called the \emph{length} of $\omega$. If $\omega$ is a
trivial or empty homotopy string, then $l(\omega)=0$. A homotopy
string $\omega$ is called \emph{direct} if all of the letters in
$\omega$ are arrows, and \emph{inverse} if all of the letters are
inverse arrows.  We define $\omega^{-1} = \alpha_1^{-1} \cdots
\alpha_l^{-1}$.

We now state when \emph{composition} of non-trivial and non-empty
homotopy strings is defined; two homotopy strings
$\omega=\alpha_l\cdots\alpha_1$ and
$\omega'=\alpha'_{l'}\cdots\alpha'_1$ where $l,l' \geq 1$, can be
composed if $s(\omega) = t(\omega')$ and one of the following
statements holds:
\begin{itemize}
\item $\alpha_1$ is direct and $\alpha'_{l'}$ is inverse and $\alpha_1^{-1}
  \neq \alpha'_{l'}$,
\item $\alpha_1$ is inverse and $\alpha'_{l'}$ is direct and $\alpha_1^{-1}
  \neq \alpha'_{l'}$,
\item $\alpha_1$ and $\alpha'_{l'}$ are both direct and
  $\alpha_1\alpha'_{l'}$ is in $I$, or
\item $\alpha_1$ and $\alpha'_{l'}$ are both inverse and
  ${\alpha'}_{l'}^{-1}\alpha_1^{-1}$ is in $I$.
\end{itemize}
The composition $\omega\cdot\omega'$ is the path $\omega\omega'$ in
$kQ'$, which is also a homotopy string.

We now define composition of homotopy strings involving trivial
homotopy strings. To do this we first need to extend the functions
$S,T$ from Definition \ref{gentlealternativ} to homotopy strings; for
any arrow $\alpha$ in $Q_1$, we define $S\alpha^{-1}=T\alpha$ and
$T\alpha^{-1}=S\alpha$. Furthermore, define $S1_x^{\varepsilon} =
\varepsilon$ and $T1_x^{\varepsilon}=-\varepsilon$ for $\varepsilon\in
\{-1,1\}$ and $x \in Q_0$. Let $\omega$ be a non-trivial and non-empty
homotopy string, then the composition $\omega \cdot 1_x^{\varepsilon}$
is defined (and equals $\omega$) if $x = s\omega$, and one of the
following statements holds:
\begin{itemize}
\item $\varepsilon = S(\alpha_1(\omega))$ and $\alpha_1(\omega)$ is an
  arrow, or
\item $\varepsilon = -S(\alpha_1(\omega))$ and $\alpha_1(\omega)$ is
  an inverse arrow.
\end{itemize}
Similarly, the composition $1_x^{\varepsilon}\cdot \omega$ is defined
(and equals $\omega$) if $x=t\omega$, and one of the following
statements holds:
\begin{itemize}
\item $\varepsilon = T(\alpha_l(\omega))$ and $\alpha_l(\omega)$ is an
  arrow, or
\item $\varepsilon = -T(\alpha_l(\omega))$ and $\alpha_l(\omega)$ is
  an inverse arrow.
\end{itemize}
The composition $1_x^{\varepsilon}\cdot 1_{x'}^{\varepsilon'}$ is
defined (and equals $1_x^{\varepsilon}$) if $x=x'$ and $\varepsilon =
\varepsilon'$.  Note that for any non-empty homotopy string $\omega$,
we have $\emptyset \cdot \omega = \omega$ and $\omega \cdot \emptyset
= \omega$.

For a homotopy string $\omega$ of positive length, there is a unique
partition $\omega=\sigma_L \cdots \sigma_1$ where each $\sigma_i$ is a
homotopy string of positive length; for each $1\leq i\leq L$ the
homotopy strings $\sigma_i$ and $\sigma_{i-1}$ can be composed as
homotopy strings; and none of the homotopy strings $\sigma_i$ can be
partitioned into non-empty non-trivial homotopy strings.  We call this
the \emph{homotopy partition} of $\omega$, and the $\sigma_j$'s are
called \emph{homotopy letters}. Define $\omega^{[i]}=\sigma_L\cdots
\sigma_{L-i+1}$ for $i>0$ and $\omega^{[0]}$ to be the trivial
homotopy string $1_{t\omega}^{\varepsilon}$ such that
$1_{t\omega}^{\varepsilon}\cdot \omega$ is defined as composition of
homotopy strings.

Furthermore, we define the \emph{degree} of $\omega$, denoted by
$\deg(\omega)$, to be the number of direct homotopy letters of
$\omega$ minus the number of inverse homotopy letters of $\omega$.
The degree of a trivial homotopy string is defined to be $0$.

A non-trivial and non-empty homotopy string $\omega=\sigma_L\cdots
\sigma_1$ with $s\omega=t\omega$ is called a \emph{homotopy band} if
$\deg(\omega)=0$; either $\sigma_{L}$ and $\sigma^{-1}_1$ are both
direct homotopy letters or $\sigma_{L}^{-1}$ and $\sigma_1$ are both
direct homotopy letters; $\sigma_1\cdot\sigma_{L}$ is defined as
composition of homotopy strings and there is no homotopy string
$\widetilde{\omega}$ with $l(\widetilde{\omega})<l(\omega)$ such that
$s\widetilde{\omega}=t\widetilde{\omega}$ and
$\omega=\widetilde{\omega}^n$ for some positive integer $n$.  If
$\omega= \sigma_L\cdots \sigma_1$ is a homotopy band, then the
homotopy string $\omega' = \sigma_{j-1}\cdots \sigma_1\sigma_L \cdots
\sigma_j$ is a \emph{rotation} of $\omega$ if $\omega'$ is a homotopy
band.

We now give an explicit description of how to construct a complex
$X_{m,\omega}$ in $\kbp$ from a homotopy string $\omega$ associated to
$\Lambda$ and an integer $m$. Let $P_x$ denote the indecomposable
projective in vertex $x$.  If $\omega = \emptyset$, then the complex
$X_{m,\omega}$ is the zero complex for all integers $m$.  If $\omega$
is trivial, that is, $\omega=1_x^{\varepsilon}$ for $\varepsilon \in
\{-1,1\}$, then the complex $X_{m,\omega}$ is the stalk complex
\[
\xymatrix{
  \cdots \ar[r] &0 \ar[r] & P_x \ar[r] &0 \ar[r] &\cdots
}
\]
with $P_x$ in degree $m$.  If $l(\omega) > 0$ we have the homotopy
partition $\omega = \sigma_L \cdots \sigma_1$ with $L \geq 1$. Let
$\sigma_i^*$ be the direct homotopy string in $\{\sigma_i,
\sigma_i^{-1}\}$.  Then $\sigma_i^*$ gives rise to a map
\[P_{t\sigma_i^*} \stackrel{p_{\sigma_i^*}}{\longrightarrow}
P_{s\sigma_i^*}~,\] sending $e_{t\sigma_i^*}$ to $\sigma_i^*
e_{s\sigma_i^*} = \sigma_i^*$, where $e_x$ is the primitive idempotent
corresponding to the vertex $x$ in $Q$.

For $m'\in\mathbb Z$ define an index set $\mathcal
I_{m'}(m,\omega)$ by
\[
\mathcal I _{m'}(m,\omega) = \{ i \in [0,L] ~|~ \deg(\omega^{[i]}) + m
= m' \} ~~.
\]
The object in degree $m'$ of $X_{m,\omega}$ is the direct sum 
\[\bigoplus_{i\in \mathcal I_{m'}(m,\omega)} P_{s (\omega^{[i]})}\ .\]
The differentials are defined componentwise, if
$\delta_X^{m'}$ is the differential from degree $m'$ to degree $m'+1$,
we define
\[
(\delta_X^{m'})_{i,j} = \left\{\begin{array}{ll}
    p_{\sigma_{L-i}^{-1}} & i = j-1 \text{ and } \sigma_{L-i} \text{ is
      inverse}\\
    p_{\sigma_{L-j}} & i = j+1 \text{ and } \sigma_{L-j} \text{ is
      direct}\\
    0 & \text{otherwise}
  \end{array}\right.
\]
for $j \in \mathcal I _{m'}(m,\omega)$ and $i \in \mathcal
I_{m'+1}(m,\omega)$.  The complexes $X_{m,\omega}$ constructed in this
way are called string complexes.  Observe that $X_{m,\omega} \cong
X_{m+\deg \omega,\omega^{-1}}$.

\begin{exmp}
  Consider the algebra $\Lambda = kQ/I$ given in Example
  \ref{ex:hovedeks}.  The homotopy string $\omega = u^{-1}cbf$ associated with
  $\Lambda$ has homotopy partition $\omega = u^{-1}
  \cdot c \cdot bf$.  We compute the string complex $X_{0,\omega}$ as
  follows:

  We have $\mathcal I_{-1}(0,\omega) = \{1\}$; $\mathcal
  I_{0}(0,\omega) = \{0,2\}$ and $\mathcal I_1(0,\omega) = \{3\}$.
  For the differentials in the complex, we get $(\delta_X^{-1})_{0,1}
  = p_{u}$; $(\delta_X^{-1})_{2,1} = p_{c}$; $(\delta_X^{0})_{3,0}
  = 0$ and $(\delta_X^{0})_{3,2} = p_{bf}$. Hence, the complex
  $X_{0,\omega}$ is
  \[
  \xymatrix{
    \cdots \ar[r] &0 \ar[r] &P_1 \ar[r]^(0.35){\left(\begin{smallmatrix}p_u \\
          p_c \end{smallmatrix}\right)} &P_{16} \oplus P_3
    \ar[r]^(0.6){\left(\begin{smallmatrix}
        0 & p_{bf} \end{smallmatrix}\right)} &P_5
    \ar[r] &0 \ar[r] &\cdots
  }
  \]
  with $P_1$ in degree $-1$.
\end{exmp}

Since homotopy bands are homotopy strings they give rise to
complexes as described above, and in addition each homotopy band
$\omega$ also gives rise to a family of band complexes
$Y_{m,\omega,\mu}$ in $\kbp$, where $m\in\mathbb Z$ and $\mu$ is an
indecomposable automorphism of a finite dimensional vector space.

Consider the equivalence relation on the set of all homotopy strings
generated by inverting a homotopy string, and let $\mathfrak W$ be a
complete class of representatives under this equivalence relation.
Similarly, we consider the equivalence relation on the set of all
homotopy bands generated by inverting a homotopy band and identifying
each homotopy band with its rotations, and let $\mathfrak B$ be a
complete set of representatives under this equivalence relation.

\begin{proposition}[{\cite[Theorem 3]{bekkert}}, see also
  {\cite[Proposition 3.1]{bob}}]
  Let $\Lambda\cong kQ/I$ be a gentle algebra.  Then the
  indecomposable objects of $\kbp$ are exactly the string complexes
  $X_{m,\omega}$ for $m \in \mathbb{Z}$ and $\omega \in \mathfrak W$,
  and the band complexes $Y_{m,\omega,\mu}$ for $m \in \mathbb{Z}$,
  $\omega \in \mathfrak B$ and $\mu$ an indecomposable automorphism of a
  finite dimensional vector space.
\end{proposition}

\subsection{Almost split triangles in $\kbp$} \label{sec:ast} Before
we state Bobi\'{n}ski's main result, giving the connection between
homotopy strings and almost split sequences, we need a result about
the almost split sequences in the category $\mathcal C$ of
indecomposable automorphisms of finite-dimensional vector spaces over
$k$. Let $\mu: V\rightarrow V$ be an indecomposable object of
$\mathcal C$ where $\dim_kV = n>0$.  Since $\mu$ is indecomposable, it
is similar to a Jordan matrix $J_n(\lambda)$ consisting of one Jordan
block, where $\lambda \in k$ is the eigenvalue of $\mu$.  Hence the
object $\mu$ of $\mathcal C$ can be represented by the pair
$(\lambda,n)$ and we denote it by $V_n(\lambda)$, as in \cite{karin}.
The following lemma from \cite{karin} gives the AR-sequence starting
in $V_n(\lambda)$ in $\mathcal C$.
\begin{lemma}
  Let $\mu=V_n(\lambda)$ be an indecomposable object in the category
  $\mathcal C$. Then there is an AR-sequence
  \[ \psi: \ \ 0\rightarrow V_n(\lambda) \rightarrow
  V_{n-1}(\lambda)\oplus V_{n+1}(\lambda) \rightarrow V_n(\lambda)
  \rightarrow 0 \] in $\mathcal C$, where $V_0(\lambda) = 0$.
\end{lemma}
In particular, the AR-structure of $\mathcal C$ consists of
homogeneous tubes, parametrized by the eigenvalues $\lambda$ of the
indecomposable automorphisms.

For each homotopy string $\omega$ we will give combinatorial
definitions of homotopy strings $\omega_+$, $\omega^+$ and
$\omega^+_+$ and integers $m'(\omega)$ and $m''(\omega)$, see Section
\ref{almostsplittriangels} and Appendix \ref{bobinskiappendix}.  With
these definitions we can state the following:

\begin{theorem}[\cite{bob}, Main Theorem]\label{bobinskimain}$ _{}$ 
\begin{itemize}
\item[i)] Let $\omega$ be a homotopy band, $m\in \mathbb{Z}$, $\lambda
  \in k$ and $\mu=V_n(\lambda)$ an indecomposable automorphism of
  a finite dimensional vector space. Then we have an almost split
  triangle in $\kbp$ of the form
\[
Y_{m,\omega,\mu} \rightarrow Y_{m,\omega,\mu_1} \oplus
Y_{m,\omega,\mu_2} \rightarrow Y_{m,\omega,\mu}\rightarrow Y_{m,\omega,\mu}[1]
\] 
where $\mu_1 = V_{n-1}(\lambda)$ and $\mu_2=V_{n+1}(\lambda)$.
\item[ii)] Let $\omega$ be a homotopy string, and $m \in
  \mathbb{Z}$. Then we have an almost split triangle in $\kbp$ of the
  form
\[
X_{m,\omega} \rightarrow X_{m+m'(\omega),\omega_+} \oplus
X_{m,\omega^+} \rightarrow X_{m+m''(\omega),\omega^+_+}\rightarrow
X_{m-1,\omega} ~~~.
\]
\end{itemize}
\end{theorem}

From now on, we will denote the string complex $X_{m,\omega}$ by
$\omega[m]$.  We call the integer $m$ the \emph{degree} of the string
complex.

\subsection{Components of the AR-quiver of
  $\kbp$}\label{sec:components}

We define the number of middle terms in an AR-triangle $\chi$ to be
$\alpha(\chi)$.  Note that from Theorem \ref{bobinskimain}, we have
$\alpha(\chi) \leq 2$ for all AR-triangles $\chi$ in $\kbp$.  By
\cite{riedtmann} we know that any component in a stable translation
quiver is of the form $\mathbb Z\Delta/G$ where $\Delta$ is a directed
tree and $G$ is an admissible group of automorphisms on
$\mathbb{Z}\Delta$.  It is clear that since any component in the
AR-quiver $\kbp$ is a stable translation quiver, and $\alpha(\chi)
\leq 2$ for all AR-triangles, $\Delta$ is either $A_n$, $A_{\infty}$ or
$A_{\infty}^{\infty}$ (see also \cite{butlerringel}).

\section{Derived equivalence classes of
  $\ant$-quivers} \label{sec:bastian}

In this section we describe representatives of the derived equivalence
classes of the cluster-tilted algebras of type $\ant$.  These algebras
are known to be gentle by \cite{abcp}.  We now
introduce some notation.

We start by fixing an embedding into the plane, to be able to make the
distinction between the clockwise and the counterclockwise
direction. 

Let $\mathcal Q_n^{\star}$ be the class of quivers of the form as in
Figure \ref{normalform}, for some non-negative integers $r_1$, $r_2$,
$s_1$, $s_2$ where $r = r_1 + r_2 > 0$ and $s = s_1 + s_2 > 0$, and $r
+ s = n$,  \cite{bastian}.  For $Q$ in $\mathcal Q_n^{\star}$ we
denote by $\Lambda_Q$ the path algebra of $Q$ modulo the ideal
generated by all zero-relations which are compositions of two arrows
in the same 3-cycle of $Q$.  We have the following:

\begin{theorem}[\cite{bastian}] \label{derivertekvivalent} $\ $
\begin{itemize}
  \item[(1)] If $\Lambda$ is a cluster-tilted algebra of type $\ant$, then
  there exists $Q$ in $\mathcal Q_n^{\star}$ such that $\Lambda$ and
  $\Lambda_Q$ are derived equivalent.

  \item[(2)] If $Q$ and $Q'$ belong to $\mathcal Q_n^{\star}$, then
  $\Lambda_Q$ and $\Lambda_{Q'}$ are derived equivalent if and only if
  they have the same parameters (up to changing the roles of $r_i$
  and $s_i$ for $i \in \{1,2\}$).
  \end{itemize}
\end{theorem}

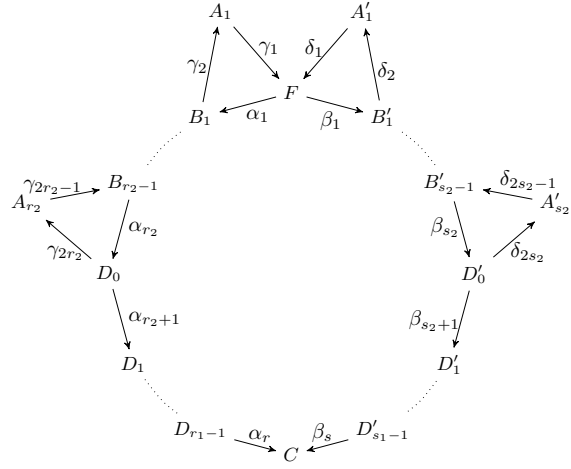
\begin{figure}[t]
\begin{minipage}{\textwidth}
\[ 
\scalebox{.8}{
\begin{tikzpicture}
  \node (1) at (30:3cm) [yvertex] {$B'_{s_2-1}$};
  \node (2) at (60:3cm) [yvertex] {$B'_1$};
  \node (3) at (90:3cm) [yvertex] {$F$};
  \node (4) at (120:3cm) [yvertex] {$B_1$};
  \node (5) at (150:3cm) [yvertex] {$B_{r_2-1}$};
  \node (6) at (180:3cm) [yvertex] {$D_0$};
  \node (7) at (210:3cm) [yvertex] {$D_1$};
  \node (8) at (240:3cm) [yvertex] {$D_{r_1-1}$};
  \node (9) at (270:3cm) [yvertex] {$C$};
  \node (10) at (300:3cm) [yvertex] {$D'_{s_1-1}$};
  \node (11) at (330:3cm) [yvertex] {$D'_1$};
  \node (12) at (360:3cm) [yvertex] {$D'_0$};
  \node (t1) at (105:4.5cm) [yvertex] {$A_1$};
  \node (t2) at (165:4.5cm) [yvertex] {$A_{r_2}$};
  \node (t3) at (75:4.5cm) [yvertex] {$A'_1$};
  \node (t4) at (15:4.5cm) [yvertex] {$A'_{s_2}$};
  \draw [->] (3)--(2);
  \draw [->] (3)--(4);
  \draw [->] (5)--(6);
  \draw [->] (6)--(7);
  \draw [->] (8)--(9);
  \draw [->] (10)--(9);
  \draw [->] (12)--(11);
  \draw [->] (1)--(12);
  \draw [dotted] (4) .. controls (135:3cm) .. (5);
  \draw [dotted] (7) .. controls (225:3cm) .. (8);
  \draw [dotted] (11) .. controls (315:3cm) .. (10);
  \draw [dotted] (2) .. controls (45:3cm) .. (1);
  \draw [->] (t1)--(3);
  \draw [->] (4)--(t1);
  \draw [->] (t2)--(5);
  \draw [->] (6)--(t2);
  \draw [->] (t3)--(3);
  \draw [->] (2)--(t3);
  \draw [->] (t4)--(1);
  \draw [->] (12)--(t4);
  \node [right] at (195:2.9cm) {$\alpha_{r_2+1}$};
  \node [above] at (260:2.95cm) { $\alpha_{r}$};
  \node [above] at (280:2.95cm) { $\beta_{s}$};
  \node [left] at (345:3cm) { $\beta_{s_2+1}$};
  \node [left] at (15:3cm) { $\beta_{s_2}$};
  \node [below] at (77:2.9cm) {$\beta_{1}$};
  \node [below] at (101:2.9cm) {$\alpha_{1}$};
  \node [right] at (165:2.9cm) { $\alpha_{r_2}$};
  \node [left] at (174:3.3cm) { $\gamma_{2r_2}$};
  \node [right] at (6:3.5cm) { $\delta_{2s_2}$};
  \node [above] at (163:4.1cm) {$\gamma_{2r_2-1}$};
  \node [above] at (18:4.1cm) {$\delta_{2s_2-1}$};
  \node [left] at (110:3.65cm) {$\gamma_{2}$};
  \node [right] at (69:3.6cm) {$\delta_{2}$};
  \node [right] at (100:3.8cm) {$\gamma_{1}$};
  \node [left] at (80:3.8cm) {$\delta_{1}$};
\end{tikzpicture} }
\]
\end{minipage}
\caption{Normal form of a quiver with all parameter values non-zero.}\label{normalform}
\end{figure}

For a quiver $Q$ in $\mathcal Q_n^{\star}$, a clockwise oriented 3-cycle in
$Q$ will be called an $\emph{r-cycle}$ and a counter-clockwise
oriented 3-cycle an $\emph{s-cycle}$. The arrows
$\alpha_{r_2+1},\ldots,\alpha_r$ are called r-arrows, and similarly
the arrows $\beta_{s_2+1},\ldots, \beta_s$ are called s-arrows.

The vertices of a quiver $Q$ in $\mathcal{Q}_n^{\star}$ can be divided into
ten disjoint sets as follows, see Figure \ref{normalform}:
\begin{itemize}
\item $A=\left\{x\in Q_0~|~ x\text{ has valency 2 and is part of an
      r-cycle} \right\}$
\item $B=\left\{x\in Q_0~|~ x\text{ is part of two r-cycles}\right\}$
\item $C=\left\{x \in Q_0 ~|~ x\text{ has valency 2 or 3, with
      $\alpha_r$ and $\beta_s$ ending here}\right\}$
\item $\widetilde C=\left\{x \in Q_0 ~|~ x\text{ has valency 2 and is
      a source}\right\}$
\item $D = \left\{x\in Q_0 ~|~ x\text{ is the starting vertex of an
      r-arrow} \right\}$
\item $F=\left\{x\in Q_0 ~|~ x \text{ has valency 4 } \alpha_1 \text{
      starting here}\right\}$
\end{itemize}
The sets $A', B'$ and $D'$ are defined similarly as $A, B$ and $D$,
respectively, by replacing r-cycle with s-cycle and r-arrow with
s-arrow. The set $F'$ is defined similarly as $F$ by replacing
``$\alpha_1$ starting here'' with ``$\alpha_r$ ending here''.  We will
use the notation $D_{\geq 1}$ for the subset $D\setminus \{D_0\}$.

The types of vertices occurring in a quiver $Q$ in
$\mathcal{Q}_n^{\star}$ depends on the values of the parameters
$r_1,r_2,s_1$ and $s_2$. In particular one should note that if there
is a vertex of type $C,\widetilde{C}, F$ or $F'$ in $Q$,
then there is only one vertex of the respective type. Also, a vertex
of type $F'$ will only occur in the case when both $r_1=0$ and
$s_1=0$. Vertices of type $B$ will only occur when $r_2>1$ and
vertices of type $D$ will only occur when $r_1>0$.
 
Moreover, as Theorem \ref{derivertekvivalent} states, the roles of
$r_i$ and $s_i$ can be interchanged while preserving the derived
equivalence. Therefore, it is sufficient to consider the cases listed
in Table \ref{parameterverdier}. Figure \ref{types} shows examples of
quivers in the cases 2--5. Case 6 is illustrated in Figure
\ref{normalform}. The first case is hereditary, and will not be
considered in this paper since its derived category is well-known
~\cite{happel}.

\begin{table}[h]
\caption{Variations of the normal form}\label{parameterverdier}
{\renewcommand{\arraystretch}{1.4}%
\begin{tabular}{|c|c|c|c|c|l|}
  \hline
  \textnumero & $r_1$ & $r_2$ & $s_1$ & $s_2$ & Possible vertices \\
  \hline
  1 & $\neq 0$ & $0$ & $\neq 0$ & $0$ & $\widetilde C, C, D, D'$  \\
  \hline
  2 & $0$ & $\neq 0$ & $0$ & $\neq 0$ & $A, A', B, B', F, F'$ \\
  \hline
  3 & $0$ & $\neq 0$ & $\neq 0$ & $0$ & $A,B,C,D'$ \\
  \hline
  4 & $\neq 0$ & $\neq 0$ & $0$ & $\neq 0$ & $A, A', B, B', C,
  D, F$ \\
  \hline
  5 & $\neq 0$ & $\neq 0$ & $\neq 0$ & $0$ & $A,B,C,D,D'$ \\
  \hline
   6 & $\neq 0$ & $\neq 0$ & $\neq 0$ & $\neq 0$ &
   $A, A',B, B',C,D,D',F$ \\
   \hline
\end{tabular}}
\end{table}

\begin{figure}[h]
  \begin{subfigure}[b]{0.22\textwidth}
    \centering
    \begin{tikzpicture}
      \node (F1) at (60:0.8cm) [tvertex] {$F$};
      \node (B) at (180:0.8cm) [tvertex] {$B_1$};
      \node (F2) at (300:0.8cm) [tvertex] {$F'$};
      \node (A1) at (120:1.45cm) [tvertex] {$A_1$};
      \node (A2) at (240:1.45cm) [tvertex] {$A_2$};
      \node (A1') at (0:1.45cm) [tvertex] {$A_1'$};
      \draw [->] (A1)--(F1);
      \draw [->] (B)--(A1);
      \draw [->] (A2)--(B);
      \draw [->] (F2)--(A2);
      \draw [->] (A1')--(F1);
      \draw [->] (F2)--(A1');
      \draw [->, bend right] (F1) to (B);
      \draw [->, bend right] (B) to (F2);
      \draw [->, bend left] (F1) to (F2);
    \end{tikzpicture}  
    \caption{\textnumero~~2}
    \label{Qtype2}
  \end{subfigure}
  \begin{subfigure}[b]{0.22\textwidth}
    \centering
      \begin{tikzpicture}
      \node (Et) at (0,0) [tvertex] {$C$};
      \node (E) at (1.5,0) [tvertex] {$D_0'$};
      \node (A) at (0.75,1.2) [tvertex] {$A_1$};
      \draw [->, bend right] (E) to (Et);
      \draw [->] (Et) to (A);
      \draw [->] (A) to (E);
      \draw [->, bend left] (E) to (Et);
    \end{tikzpicture}
    \caption{\textnumero~~3}
    \label{Qtype3}
  \end{subfigure}
  \begin{subfigure}[b]{0.22\textwidth}
    \centering
    \begin{tikzpicture}
      \node (F) at (90:0.8cm) [tvertex] {$F$};
      \node (E) at (180:0.8cm) [tvertex] {$D_0$};
      \node (D) at (270:0.8cm) [tvertex] {$D_1$};
      \node (Et) at (0:0.8cm) [tvertex] {$C$};
      \node (A) at (135:1.45cm) [tvertex] {$A_1$};
      \node (A') at (45:1.45cm) [tvertex] {$A_1'$};
      \draw [->] (A)--(F);
      \draw [->] (E)--(A);
      \draw [->] (A')--(F);
      \draw [->] (Et)--(A');
      \draw [->, bend right] (F) to (E);
      \draw [->, bend right] (E) to (D);
      \draw [->, bend left] (F) to (Et);
      \draw [->, bend right] (D) to (Et);
    \end{tikzpicture}
    \caption{\textnumero~~4}
    \label{Qtype4}
  \end{subfigure}
  \begin{subfigure}[b]{0.22\textwidth}
    \centering
    \begin{tikzpicture}
      \node (E') at (90:0.8cm) [tvertex] {$D_0'$};
      \node (E) at (180:0.8cm) [tvertex] {$D_0$};
      \node (C) at (270:0.8cm) [tvertex] {$C$};
      \node (D) at (0:0.8cm) [tvertex] {$D_1'$};
      \node (A) at (135:1.45cm) [tvertex] {$A_1$};
      \draw [->] (A)--(E');
      \draw [->] (E)--(A);
      \draw [->, bend right] (E') to (E);
      \draw [->, bend right] (E) to (C);
      \draw [->, bend left] (E') to (D);
      \draw [->, bend left] (D) to (C);
    \end{tikzpicture}
    \caption{\textnumero~~5}
    \label{Qtype5}
  \end{subfigure}
  \caption{Some quivers in $\mathcal{Q}_n^{\star}$.}\label{types}
\end{figure}
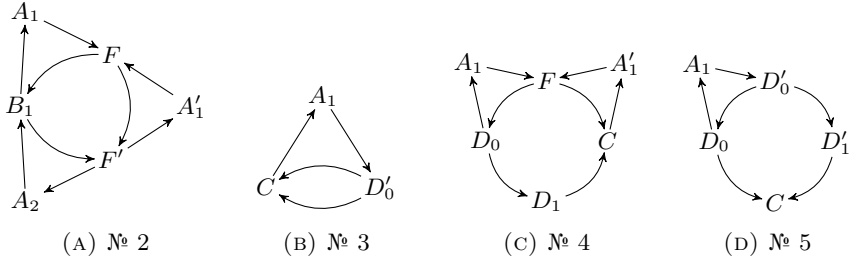

\section{{\rm r}-walks and {\rm s}-walks}
In this section we define r- and s-walks, which are special ways of
traversing a quiver in $\mathcal Q_n^{\star}$. We will also define
reduction of homotopy strings.  The walks and reductions will later be
used to describe AR-triangles in $\kbp$.

Let $Q$ be a quiver in $\mathcal Q_n^{\star}$.  For a vertex $x\in
Q_0$ let a \textit{walk starting in $x$} be a possibly infinite series
of homotopy strings $[w_1, w_2, \ldots]$ with $sw_1 = x$, $sw_i
=tw_{i-1}$ for $i>1$ and such that $w_n\cdots w_2w_1$ for any $n>1$ is
a homotopy string. However, it is not necessary that $w_i\cdot
w_{i-1}$ is defined as composition of homotopy strings.

We define a clockwise r-walk $W = [w_1, w_2, \ldots]$ starting in a
vertex $x$ in $A\cup D$: It is recursively defined by the
function $cw\_r(x)$, where if $x$ is the vertex
\begin{itemize}
\item $A_i$ for $i> 1$: then $cw\_r(x) =
  \gamma_{2(i-1)}\gamma_{2i-1}$, going from $A_i$ to $A_{i-1}$
\item $A_1$: $\left\{ \begin{array}{ll} r_1 = 0 & \text{then
        $cw\_r(x) = \gamma_{2r_2}\beta_s\cdots \beta_1\gamma_1$, going from $A_1$ to $A_{r_2}$ } \\
      r_1 > 0 & \text{then $cw\_r(x) = \alpha_r^{-1} \beta_s
        \cdots \beta_1 \gamma_1$, going from $A_1$ to $D_{r_1-1}$} \\
    \end{array} \right.$
\item $D_i$ for $i>0$: then $cw\_r(x) = \alpha_{r_2 +i}^{-1}$, going
  from $D_i$ to $D_{i-1}$
\item $D_0$: $\left\{ \begin{array}{ll} r_2 = 0 & \text{then $cw\_r(x)
        =
        \alpha_r^{-1}\beta_s\cdots \beta_1$, going from $D_0$ to $D_{r_1-1}$} \\
      r_2 > 0 & \text{then $cw\_r(x) = \gamma_{2r_2}$, going from $D_0$ to $A_{r_2}$} \\
		\end{array} \right.$
\end{itemize}
where the vertices $A_i$ and $D_i$ are as shown in Figure
\ref{normalform}. We define $w_1 = cw\_r(x)$. Observe that the vertex
$t(w_1)$ is always in $A\cup D$. Further, we define the $i$th step of
the clockwise r-walk to be $w_i = cw\_r(t(w_{i-1}))$ for
$i>1$. Observe that $w_i = w_{i+r}$ for $i \geq 1$.  Note that
$cw\_r(x)$ is always the shortest homotopy string from $x$ to $t(w_1)$
in the clockwise direction.

A clockwise r-walk is illustrated in Figure \ref{cwr}. The vertices of
type $A \cup D$ are marked with $\star$, and the paths between the
$\star$'s in the clockwise direction are the steps of the walk.

Now we extend the function $cw\_r(x)$ to vertices $x$ of type
$B,B',C,D',F$ or $F'$. The \emph{clockwise r-prefix} is the shortest
clockwise homotopy string $w$ with $s(w) = x$, such that there exists
some homotopy string $w'$ such that $ww' = cw\_r(y)$ for some vertex
$y \in A\cup D$. From the above, it follows that $w$ is unique, and we
denote it by $cw\_r\_p(x)$.  The extended definition of the clockwise
r-walk is then
\[
cw\_r(x) = \begin{cases}
  cw\_r(x) & x \in A \cup D \\
  cw\_r\_p(x) & x \in B \cup B' \cup C \cup D' \cup F \cup F'
  \end{cases}
\]

Note that according to Table \ref{parameterverdier}, we always have
$r_2 > 0$.  However, we have included the case $r_2 = 0$ in the
definition of clockwise r-walk, as this is needed to give a complete
definition of counter-clockwise s-walk which will be defined as a
mirror image of the clockwise r-walk.

Next we define a \emph{counter-clockwise r-walk} $V =
[v_1,v_2,\ldots]$ starting in a vertex $x$ in $A \cup C \cup D_{\geq
  1}$. It is recursively defined by the function $ccw\_r(x)$, given by
\begin{itemize}
\item $A_i$ for $i < r_2$: then $ccw\_r(x) =
  \gamma_{(2(i+1)-1)}^{-1}\gamma_{2i}^{-1}$, going from $A_i$ to $A_{i+1}$
\item $A_{r_2}$: $\left\{ \begin{array}{ll}
      r_1 = 0 & \text{then $ccw\_r(x)
        =\gamma_1^{-1}\beta_1^{-1}\cdots\beta_{s}^{-1}\gamma_{2r_2}^{-1}$,
        going from $A_{r_2}$ to $A_1$} \\
      r_1 > 0 & \text{then  $ccw\_r(x) =
        \alpha_{r_2+1}\gamma_{2r_2}^{-1}$,}\\
      & \text{going from $A_{r_2}$ to $D_1$, or from $A_{r_2}$ to $C$ } \\
    \end{array} \right.$
\item $C$: $\left\{ \begin{array}{ll} r_2 = 0 & \text{then $ccw\_r(x)
        =
        \alpha_{1}\beta_{1}^{-1}\cdots \beta_s^{-1}$,}\\ &\text{going from $C$ to $C$, or from $C$ to $D_1$} \\
      r_2 > 0 & \text{then, $ccw\_r(x) =
        \gamma_1^{-1}\beta_1^{-1}\cdots \beta_s^{-1}$, going from $C$ to $A_{1}$} \\
    \end{array} \right.$
\item $D_i$ for $1 \leq i \leq r_1 -1$: then $ccw\_r(x) = \alpha_{r_2
    + i + 1}$, going from $D_i$ to $D_{i+1}$, or from $D_i$ to $C$
\end{itemize}
As for the clockwise case, we define $v_1 = ccw\_r(x)$, and the $i$th
step of the counter-clockwise r-walk is defined by $v_i =
ccw\_r(t(v_{i-1}))$ for $i > 1$. See Figure \ref{ccwr} for an
illustration of the steps in a counter-clockwise r-walk.

The \emph{counter-clockwise r-prefix} for a vertex $x$ in
$B,B',D_0,D',F$ or $F'$ is defined as follows: It is the shortest
counter-clockwise homotopy string $v$ with $s(v) = x$, such that there
exists some homotopy string $v'$ such that $vv' = ccw\_r(y)$ for some
vertex $y \in A\cup C \cup D_{\geq 1}$. Again, $v$ is
unique, and we denote it by $ccw\_r\_p(x)$.  We extend the
counter-clockwise r-walk as follows:
\[
ccw\_r(x) = \begin{cases}
  ccw\_r(x) & x \in A \cup C \cup D_{\geq 1} \\
  ccw\_r\_p(x) & x \in B \cup B' \cup D_0 \cup D' \cup F \cup F'
  \end{cases}
\]

For the same reason as for the clockwise r-walk, we have included
the case when $r_2 = 0$.

A \emph{counter-clockwise s-walk} is defined to be the mirror image of
a clockwise r-walk, see Figure \ref{ccws}.  A \emph{clockwise s-walk}
is defined as the mirror image of a counter-clockwise r-walk, see
Figure \ref{cws}.

\begin{figure}[t]
  \centering
  \begin{subfigure}[b]{0.46\textwidth}
    \begin{tikzpicture}
      \node (1) at (90:2.25cm) [nvertex] {};
      \node (2) at (110:2.25cm) [nvertex] {};
      \node (3) at (130:2.25cm) [nvertex] {};
      \node (4) at (150:2.25cm) [nvertex] {};
      \node (5) at (170:2.25cm) [nvertex] {};
      \node (6) at (190:2.25cm) [nvertex] {};
      \node (7) at (210:2.25cm) [tvertex] {$\star$};
      \node (8) at (230:2.25cm) [tvertex] {$\star$};
      \node (9) at (250:2.25cm) [tvertex] {$\star$};
      \node (10) at (270:2.25cm) [tvertex] {$\star$};
      \node (11) at (290:2.25cm) [tvertex] {$\star$};
      \node (12) at (310:2.25cm) [tvertex] {$\star$};
      \node (13) at (330:2.25cm) [nvertex] {};
      \node (14) at (350:2.25cm) [nvertex] {};
      \node (15) at (10:2.25cm) [nvertex] {};
      \node (16) at (30:2.25cm) [nvertex] {};
      \node (17) at (50:2.25cm) [nvertex] {};
      \node (18) at (70:2.25cm) [nvertex] {};
      \node (19) at (100:3cm) [tvertex] {$\star$};
      \node (20) at (140:3cm) [tvertex] {$\star$};
      \node (21) at (180:3cm) [tvertex] {$\star$};
      \node (22) at (200:3cm) [tvertex] {$\star$};
      \node (23) at (40:3cm) [nvertex] {};
      \node (24) at (80:3cm) [nvertex] {};
      \draw [->, snake=coil, segment amplitude=1pt, segment length =
      3.5pt, segment aspect=0,>=latex, line after snake=2.5pt] (19)--(1);
      \draw [->] (1)--(2);
      \draw [->, snake=coil, segment amplitude=1pt, segment length =
      3.5pt, segment aspect=0,>=latex, line after snake=2.5pt] (2)--(19);
      \draw [->, snake=coil, segment amplitude=1pt, segment length =
      3.5pt, segment aspect=0,>=latex, line after snake=2.5pt] (20)--(3);
      \draw [->] (3)--(4);
      \draw [->, snake=coil, segment amplitude=1pt, segment length =
      3.5pt, segment aspect=0,>=latex, line after snake=2.5pt] (4)--(20);
      \draw [->, snake=coil, segment amplitude=1pt, segment length =
      3.5pt, segment aspect=0,>=latex, line after snake=2.5pt] (21)--(5);
      \draw [->] (5)--(6);
      \draw [->, snake=coil, segment amplitude=1pt, segment length =
      3.5pt, segment aspect=0,>=latex, line after snake=2.5pt] (6)--(21);
      \draw [->, snake=coil, segment amplitude=1pt, segment length =
      3.5pt, segment aspect=0,>=latex, line after snake=2.5pt] (22)--(6);
      \draw [->] (6)--(7);
      \draw [->, snake=coil, segment amplitude=1pt, segment length =
      3.5pt, segment aspect=0,>=latex, line after snake=2.5pt] (7)--(22);
      \draw [->] (23)--(17);
      \draw [->, snake=coil, segment amplitude=1pt, segment length =
      3.5pt, segment aspect=0,>=latex, line after snake=2.5pt] (17)--(16);
      \draw [->] (16)--(23);
      \draw [->] (24)--(1);
      \draw [->, snake=coil, segment amplitude=1pt, segment length =
      3.5pt, segment aspect=0,>=latex, line after snake=2.5pt] (1)--(18);
      \draw [->] (18)--(24);
      \draw [->, snake=coil, segment amplitude=1pt, segment length =
      3.5pt, segment aspect=0,>=latex, line after snake=2.5pt] (7)--(8);
      \draw [->, snake=coil, segment amplitude=1pt, segment length =
      3.5pt, segment aspect=0,>=latex, line after snake=2.5pt] (9)--(10);
      \draw [->, snake=coil, segment amplitude=1pt, segment length =
      3.5pt, segment aspect=0,>=latex, line after snake=2.5pt] (11)--(12);
      \draw [->, snake=coil, segment amplitude=1pt, segment length =
      3.5pt, segment aspect=0,>=latex, line after snake=2.5pt] (12)--(13);
      \draw [->, snake=coil, segment amplitude=1pt, segment length =
      3.5pt, segment aspect=0,>=latex, line after snake=2.5pt] (14)--(13);
      \draw [->, snake=coil, segment amplitude=1pt, segment length =
      3.5pt, segment aspect=0,>=latex, line after snake=2.5pt] (16)--(15);
      \draw [dotted] (2)--(3);
      \draw [dotted] (5)--(4);
      \draw [dotted] (8)--(9);
      \draw [dotted] (10)--(11);
      \draw [dotted] (15)--(14);
      \draw [dotted] (18)--(17);
      \draw [->] (60:1.5cm) .. controls (20:1.7cm) and (-20:1.7cm) .. (-60:1.5cm);
    \end{tikzpicture}
    \caption{A clockwise r-walk.}
    \label{cwr}
  \end{subfigure}
  \begin{subfigure}[b]{0.46\textwidth}
    \begin{tikzpicture}
      \node (1) at (90:2.25cm) [nvertex] {};
      \node (2) at (110:2.25cm) [nvertex] {};
      \node (3) at (130:2.25cm) [nvertex] {};
      \node (4) at (150:2.25cm) [nvertex] {};
      \node (5) at (170:2.25cm) [nvertex] {};
      \node (6) at (190:2.25cm) [nvertex] {};
      \node (7) at (210:2.25cm) [nvertex] {};
      \node (8) at (230:2.25cm) [tvertex] {$\star$};
      \node (9) at (250:2.25cm) [tvertex] {$\star$};
      \node (10) at (270:2.25cm) [tvertex] {$\star$};
      \node (11) at (290:2.25cm) [tvertex] {$\star$};
      \node (12) at (310:2.25cm) [tvertex] {$\star$};
      \node (13) at (330:2.25cm) [tvertex] {$\star$};
      \node (14) at (350:2.25cm) [nvertex] {};
      \node (15) at (10:2.25cm) [nvertex] {};
      \node (16) at (30:2.25cm) [nvertex] {};
      \node (17) at (50:2.25cm) [nvertex] {};
      \node (18) at (70:2.25cm) [nvertex] {};
      \node (19) at (100:3cm) [tvertex] {$\star$};
      \node (20) at (120:3cm) [tvertex] {$\star$};
      \node (21) at (160:3cm) [tvertex] {$\star$};
      \node (22) at (200:3cm) [tvertex] {$\star$};
      \node (23) at (40:3cm) [nvertex] {};
      \node (24) at (80:3cm) [nvertex] {};
      \draw [->, snake=coil, segment amplitude=1pt, segment length =
      3.5pt, segment aspect=0,>=latex, line after snake=2.5pt] (19)--(1);
      \draw [->] (1)--(2);
      \draw [->, snake=coil, segment amplitude=1pt, segment length =
      3.5pt, segment aspect=0,>=latex, line after snake=2.5pt] (2)--(19);
      \draw [->, snake=coil, segment amplitude=1pt, segment length =
      3.5pt, segment aspect=0,>=latex, line after snake=2.5pt] (20)--(2);
      \draw [->] (2)--(3);
      \draw [->, snake=coil, segment amplitude=1pt, segment length =
      3.5pt, segment aspect=0,>=latex, line after snake=2.5pt] (3)--(20);
      \draw [->, snake=coil, segment amplitude=1pt, segment length =
      3.5pt, segment aspect=0,>=latex, line after snake=2.5pt] (21)--(4);
      \draw [->] (4)--(5);
      \draw [->, snake=coil, segment amplitude=1pt, segment length =
      3.5pt, segment aspect=0,>=latex, line after snake=2.5pt] (5)--(21);
      \draw [->, snake=coil, segment amplitude=1pt, segment length =
      3.5pt, segment aspect=0,>=latex, line after snake=2.5pt] (22)--(6);
      \draw [->] (6)--(7);
      \draw [->, snake=coil, segment amplitude=1pt, segment length =
      3.5pt, segment aspect=0,>=latex, line after snake=2.5pt] (7)--(22);
      \draw [->] (23)--(17);
      \draw [->, snake=coil, segment amplitude=1pt, segment length =
      3.5pt, segment aspect=0,>=latex, line after snake=2.5pt] (17)--(16);
      \draw [->] (16)--(23);
      \draw [->] (24)--(1);
      \draw [->, snake=coil, segment amplitude=1pt, segment length =
      3.5pt, segment aspect=0,>=latex, line after snake=2.5pt] (1)--(18);
      \draw [->] (18)--(24);
      \draw [->, snake=coil, segment amplitude=1pt, segment length =
      3.5pt, segment aspect=0,>=latex, line after snake=2.5pt] (7)--(8);
      \draw [->, snake=coil, segment amplitude=1pt, segment length =
      3.5pt, segment aspect=0,>=latex, line after snake=2.5pt] (8)--(9);
      \draw [->, snake=coil, segment amplitude=1pt, segment length =
      3.5pt, segment aspect=0,>=latex, line after snake=2.5pt] (10)--(11);
      \draw [->, snake=coil, segment amplitude=1pt, segment length =
      3.5pt, segment aspect=0,>=latex, line after snake=2.5pt] (12)--(13);
      \draw [->, snake=coil, segment amplitude=1pt, segment length =
      3.5pt, segment aspect=0,>=latex, line after snake=2.5pt] (14)--(13);
      \draw [->, snake=coil, segment amplitude=1pt, segment length =
      3.5pt, segment aspect=0,>=latex, line after snake=2.5pt] (16)--(15);
      \draw [dotted] (3)--(4);
      \draw [dotted] (6)--(5);
      \draw [dotted] (9)--(10);
      \draw [dotted] (11)--(12);
      \draw [dotted] (15)--(14);
      \draw [dotted] (18)--(17);
      \draw [<-] (60:1.5cm) .. controls (20:1.7cm) and (-20:1.7cm) .. (-60:1.5cm);
    \end{tikzpicture}
    \caption{A counter-clockwise r-walk.}
    \label{ccwr}
  \end{subfigure}

\vspace{10pt}

  \begin{subfigure}[b]{0.46\textwidth}
    \reflectbox{
    \begin{tikzpicture}
      \node (1) at (90:2.25cm) [nvertex] {};
      \node (2) at (110:2.25cm) [nvertex] {};
      \node (3) at (130:2.25cm) [nvertex] {};
      \node (4) at (150:2.25cm) [nvertex] {};
      \node (5) at (170:2.25cm) [nvertex] {};
      \node (6) at (190:2.25cm) [nvertex] {};
      \node (7) at (210:2.25cm) [nvertex] {};
      \node (8) at (230:2.25cm) [tvertex] {$\star$};
      \node (9) at (250:2.25cm) [tvertex] {$\star$};
      \node (10) at (270:2.25cm) [tvertex] {$\star$};
      \node (11) at (290:2.25cm) [tvertex] {$\star$};
      \node (12) at (310:2.25cm) [tvertex] {$\star$};
      \node (13) at (330:2.25cm) [tvertex] {$\star$};
      \node (14) at (350:2.25cm) [nvertex] {};
      \node (15) at (10:2.25cm) [nvertex] {};
      \node (16) at (30:2.25cm) [nvertex] {};
      \node (17) at (50:2.25cm) [nvertex] {};
      \node (18) at (70:2.25cm) [nvertex] {};
      \node (19) at (100:3cm) [tvertex] {$\star$};
      \node (20) at (120:3cm) [tvertex] {$\star$};
      \node (21) at (160:3cm) [tvertex] {$\star$};
      \node (22) at (200:3cm) [tvertex] {$\star$};
      \node (23) at (40:3cm) [nvertex] {};
      \node (24) at (80:3cm) [nvertex] {};
      \draw [->, snake=coil, segment amplitude=1pt, segment length =
      3.5pt, segment aspect=0,>=latex, line after snake=2.5pt] (19)--(1);
      \draw [->] (1)--(2);
      \draw [->, snake=coil, segment amplitude=1pt, segment length =
      3.5pt, segment aspect=0,>=latex, line after snake=2.5pt] (2)--(19);
      \draw [->, snake=coil, segment amplitude=1pt, segment length =
      3.5pt, segment aspect=0,>=latex, line after snake=2.5pt] (20)--(2);
      \draw [->] (2)--(3);
      \draw [->, snake=coil, segment amplitude=1pt, segment length =
      3.5pt, segment aspect=0,>=latex, line after snake=2.5pt] (3)--(20);
      \draw [->, snake=coil, segment amplitude=1pt, segment length =
      3.5pt, segment aspect=0,>=latex, line after snake=2.5pt] (21)--(4);
      \draw [->] (4)--(5);
      \draw [->, snake=coil, segment amplitude=1pt, segment length =
      3.5pt, segment aspect=0,>=latex, line after snake=2.5pt] (5)--(21);
      \draw [->, snake=coil, segment amplitude=1pt, segment length =
      3.5pt, segment aspect=0,>=latex, line after snake=2.5pt] (22)--(6);
      \draw [->] (6)--(7);
      \draw [->, snake=coil, segment amplitude=1pt, segment length =
      3.5pt, segment aspect=0,>=latex, line after snake=2.5pt] (7)--(22);
      \draw [->] (23)--(17);
      \draw [->, snake=coil, segment amplitude=1pt, segment length =
      3.5pt, segment aspect=0,>=latex, line after snake=2.5pt] (17)--(16);
      \draw [->] (16)--(23);
      \draw [->] (24)--(1);
      \draw [->, snake=coil, segment amplitude=1pt, segment length =
      3.5pt, segment aspect=0,>=latex, line after snake=2.5pt] (1)--(18);
      \draw [->] (18)--(24);
      \draw [->, snake=coil, segment amplitude=1pt, segment length =
      3.5pt, segment aspect=0,>=latex, line after snake=2.5pt] (7)--(8);
      \draw [->, snake=coil, segment amplitude=1pt, segment length =
      3.5pt, segment aspect=0,>=latex, line after snake=2.5pt] (8)--(9);
      \draw [->, snake=coil, segment amplitude=1pt, segment length =
      3.5pt, segment aspect=0,>=latex, line after snake=2.5pt] (10)--(11);
      \draw [->, snake=coil, segment amplitude=1pt, segment length =
      3.5pt, segment aspect=0,>=latex, line after snake=2.5pt] (12)--(13);
      \draw [->, snake=coil, segment amplitude=1pt, segment length =
      3.5pt, segment aspect=0,>=latex, line after snake=2.5pt] (14)--(13);
      \draw [->, snake=coil, segment amplitude=1pt, segment length =
      3.5pt, segment aspect=0,>=latex, line after snake=2.5pt] (16)--(15);
      \draw [dotted] (3)--(4);
      \draw [dotted] (6)--(5);
      \draw [dotted] (9)--(10);
      \draw [dotted] (11)--(12);
      \draw [dotted] (15)--(14);
      \draw [dotted] (18)--(17);
      \draw [->] (120:1.5cm) .. controls (160:1.7cm) and (-160:1.7cm) .. (-120:1.5cm);
    \end{tikzpicture}}
    \caption{A clockwise s-walk.}
    \label{cws}
  \end{subfigure}
  \begin{subfigure}[b]{0.46\textwidth}
    \centering
    \reflectbox{
    \begin{tikzpicture}
      \node (1) at (90:2.25cm) [nvertex] {};
      \node (2) at (110:2.25cm) [nvertex] {};
      \node (3) at (130:2.25cm) [nvertex] {};
      \node (4) at (150:2.25cm) [nvertex] {};
      \node (5) at (170:2.25cm) [nvertex] {};
      \node (6) at (190:2.25cm) [nvertex] {};
      \node (7) at (210:2.25cm) [tvertex] {$\star$};
      \node (8) at (230:2.25cm) [tvertex] {$\star$};
      \node (9) at (250:2.25cm) [tvertex] {$\star$};
      \node (10) at (270:2.25cm) [tvertex] {$\star$};
      \node (11) at (290:2.25cm) [tvertex] {$\star$};
      \node (12) at (310:2.25cm) [tvertex] {$\star$};
      \node (13) at (330:2.25cm) [nvertex] {};
      \node (14) at (350:2.25cm) [nvertex] {};
      \node (15) at (10:2.25cm) [nvertex] {};
      \node (16) at (30:2.25cm) [nvertex] {};
      \node (17) at (50:2.25cm) [nvertex] {};
      \node (18) at (70:2.25cm) [nvertex] {};
      \node (19) at (100:3cm) [tvertex] {$\star$};
      \node (20) at (140:3cm) [tvertex] {$\star$};
      \node (21) at (180:3cm) [tvertex] {$\star$};
      \node (22) at (200:3cm) [tvertex] {$\star$};
      \node (23) at (40:3cm) [nvertex] {};
      \node (24) at (80:3cm) [nvertex] {};
      \draw [->, snake=coil, segment amplitude=1pt, segment length =
      3.5pt, segment aspect=0,>=latex, line after snake=2.5pt] (19)--(1);
      \draw [->] (1)--(2);
      \draw [->, snake=coil, segment amplitude=1pt, segment length =
      3.5pt, segment aspect=0,>=latex, line after snake=2.5pt] (2)--(19);
      \draw [->, snake=coil, segment amplitude=1pt, segment length =
      3.5pt, segment aspect=0,>=latex, line after snake=2.5pt] (20)--(3);
      \draw [->] (3)--(4);
      \draw [->, snake=coil, segment amplitude=1pt, segment length =
      3.5pt, segment aspect=0,>=latex, line after snake=2.5pt] (4)--(20);
      \draw [->, snake=coil, segment amplitude=1pt, segment length =
      3.5pt, segment aspect=0,>=latex, line after snake=2.5pt] (21)--(5);
      \draw [->] (5)--(6);
      \draw [->, snake=coil, segment amplitude=1pt, segment length =
      3.5pt, segment aspect=0,>=latex, line after snake=2.5pt] (6)--(21);
      \draw [->, snake=coil, segment amplitude=1pt, segment length =
      3.5pt, segment aspect=0,>=latex, line after snake=2.5pt] (22)--(6);
      \draw [->] (6)--(7);
      \draw [->, snake=coil, segment amplitude=1pt, segment length =
      3.5pt, segment aspect=0,>=latex, line after snake=2.5pt] (7)--(22);
      \draw [->] (23)--(17);
      \draw [->, snake=coil, segment amplitude=1pt, segment length =
      3.5pt, segment aspect=0,>=latex, line after snake=2.5pt] (17)--(16);
      \draw [->] (16)--(23);
      \draw [->] (24)--(1);
      \draw [->, snake=coil, segment amplitude=1pt, segment length =
      3.5pt, segment aspect=0,>=latex, line after snake=2.5pt] (1)--(18);
      \draw [->] (18)--(24);
      \draw [->, snake=coil, segment amplitude=1pt, segment length =
      3.5pt, segment aspect=0,>=latex, line after snake=2.5pt] (7)--(8);
      \draw [->, snake=coil, segment amplitude=1pt, segment length =
      3.5pt, segment aspect=0,>=latex, line after snake=2.5pt] (9)--(10);
      \draw [->, snake=coil, segment amplitude=1pt, segment length =
      3.5pt, segment aspect=0,>=latex, line after snake=2.5pt] (11)--(12);
      \draw [->, snake=coil, segment amplitude=1pt, segment length =
      3.5pt, segment aspect=0,>=latex, line after snake=2.5pt] (12)--(13);
      \draw [->, snake=coil, segment amplitude=1pt, segment length =
      3.5pt, segment aspect=0,>=latex, line after snake=2.5pt] (14)--(13);
      \draw [->, snake=coil, segment amplitude=1pt, segment length =
      3.5pt, segment aspect=0,>=latex, line after snake=2.5pt] (16)--(15);
      \draw [dotted] (2)--(3);
      \draw [dotted] (5)--(4);
      \draw [dotted] (8)--(9);
      \draw [dotted] (10)--(11);
      \draw [dotted] (15)--(14);
      \draw [dotted] (18)--(17);
      \draw [<-] (120:1.5cm) .. controls (160:1.7cm) and (-160:1.7cm) .. (-120:1.5cm);
    \end{tikzpicture}}
    \caption{A counter-clockwise s-walk.}
    \label{ccws}
  \end{subfigure}
  \caption{}\label{fig:}
  \label{walks}
\end{figure}

\subsection{Reduction of a homotopy string}\label{redsubkap}
Let $\omega$ be a non-trivial and non-empty homotopy string with one
of the following properties:
\begin{enumerate}[(i)]
  \item $t(\omega) \in A \cup D$, and such that
    $\alpha_l(\omega)$ is the last letter in the $r$th step of the
    clockwise r-walk starting in $t(\omega)$,\label{condition1}
  \item $t(\omega) \in A' \cup D'$, and such that
    $\alpha_l(\omega)$ is the last letter in the $s$th step of the
    counter-clockwise s-walk starting in $t(\omega)$,\label{condition2}
  \item $t(\omega) \in A \cup C \cup D_{\geq 1}$, and such that
    $\alpha_l(\omega)$ is the last letter in the $r$th step of the
    counter-clockwise r-walk starting in $t(\omega)$,\label{condition3}
  \item $t(\omega) \in A' \cup C \cup D_{\geq 1}'$, and such that
    $\alpha_l(\omega)$ is the last letter in the $s$th step of the
    clockwise s-walk starting in $t(\omega)$.\label{condition4}
\end{enumerate}
Observe that a homotopy string $\omega$ satisfies at most one of these
properties.

\begin{definition}\label{reduction1}
  Let $\omega$ be a homotopy string satisfying property (i).  Let
  $w_r$ be the $r$th step of the clockwise r-walk starting
  in $t(\omega)$.  We define the \emph{clockwise r-reduction} of
  $\omega$ to be $\omega'$, where $\omega = \sigma \omega'$ for a
  non-trivial homotopy string $\sigma$ satisfying
  \begin{itemize}
  \item $w_r = \sigma \hat{w}$ for some homotopy string $\hat{w}$, and
  \item there is no $\sigma'$ such that $\omega = \sigma' \omega''$
    and $w_r = \sigma' \widetilde{w}$ with $l(\sigma') > l(\sigma)$.
  \end{itemize}
\end{definition}
Similarly, we define \emph{counter-clockwise r-reduction} by replacing
property (i) by property (iii), and by letting $w_r$ be the $r$th step
of the counter-clockwise r-walk starting in $t(\omega)$.  The
\emph{clockwise s-reduction} and \emph{counter-clockwise s-reduction}
are defined analogously.

Note that for some homotopy strings $\omega$, the reduction
removes $\omega$ itself -- that is, $\omega = \sigma \omega'$ where
$\sigma = \omega$. In this case, $\omega'$ is the trivial homotopy
string $1_{s\omega}^{\varepsilon}$ such that $\omega \cdot
1_{s\omega}^{\varepsilon}$ is defined as composition of homotopy
strings.  To do this, we fix the string functions $S$ and $T$
described in Definition \ref{gentlealternativ}.  See Appendix
\ref{SandT} for details.

\begin{exmp}
  Recall Example \ref{ex:hovedeks}. Consider the homotopy string
  $\omega = bfed$ associated with $\Lambda$.  It is clear that
  $\omega$ satisfies property (i).  Then the homotopy string $ed$ is
  the clockwise r-reduction of $\omega$.  Note that this homotopy
  string also satisfies property (i), and has the clockwise
  r-reduction $d$.  In the last case, the homotopy string we remove
  is the clockwise r-prefix for vertex $4$.

  Moreover, the homotopy string $\nu = jgd$ satisfies property (ii),
  and has counter-clockwise r-reduction $gd$.  Here, the homotopy
  string we remove is the counter-clockwise r-prefix for vertex $6$.
  Note that the homotopy string $gd$ does not satisfy properties
  (i)--(iv) and hence does not have a reduction.
\end{exmp}
\section{Almost split triangles for string
  complexes}\label{almostsplittriangels}
Let $\Lambda$ be a cluster-tilted algebra of type $\ant$, and let $\omega[m]$
be a string complex in $\kbp$. In this section, we give explicit
calculations of the AR-triangle starting in $\omega[m]$ and
the AR-triangle ending in $\omega[m]$. 

Theorem \ref{bobinskimain} states that the almost split sequence
starting in $\omega[m]$ is of the form
\[
\omega[m] \rightarrow \omega^+[m+m'(\omega)]\oplus
\omega_+[m] \rightarrow \omega^+_+[m+m''(\omega)] \rightarrow \omega[m-1]
\]
where all the involved homotopy strings and integers can be found
combinatorially by using Bobi\'nski's algorithm (see Appendix
\ref{bobinskiappendix}).  Figure \ref{omeganeighbours} shows the
AR-triangle ending in $\omega[m]$, and the AR-triangle starting in
$\omega[m]$.

\begin{figure}[h!]
\[
\begin{tikzpicture}[scale=1.5]
  \node (omega) at (0,0) [yvertex] {$\omega[m]$};
  \node (+omega) at (1,1) [yvertex] {$\omega^+[m+m'(\omega)]$};
  \node (omega+) at (1,-1) [yvertex] {$\omega_+[m]$};
  \node (-omega) at (-1,1) [yvertex] {$\omega^-[m]$};
  \node (omega-) at (-1,-1) [yvertex] {$\omega_-[m-m'(\omega_-)]$};
  \node (+omega+) at (2,0) [yvertex] {$\omega^+_+[m+m''(\omega)]$};
  \node (-omega-) at (-2,0) [yvertex] {${}\omega^-_-[m-m''(\omega^-_-)]$};
  \draw [->] (omega)--(+omega);
  \draw [->] (omega)--(omega+);
  \draw [->] (-omega)--(omega);
  \draw [->] (omega-)--(omega);
  \draw [->] (omega+)--(+omega+);
  \draw [->] (+omega)--(+omega+);
  \draw [->] (-omega-)--(omega-);
  \draw [->] (-omega-)--(-omega);
  \draw [dotted, ->] (+omega)--(-omega);
  \draw [dotted, ->] (omega+)--(omega-);
  \draw [dotted, ->] (omega)--(-omega-);
  \draw [dotted, ->] (+omega+)--(omega);
\end{tikzpicture}
\]
\caption{The AR-triangles starting and ending in $\omega[m]$.}\label{omeganeighbours}
\end{figure}
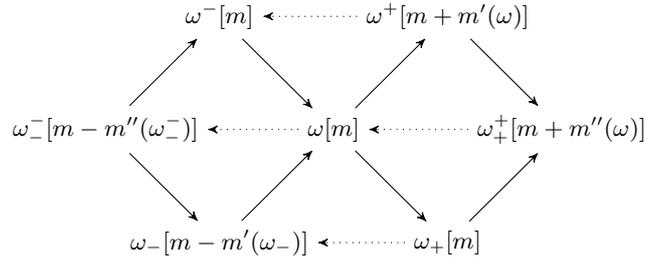

\begin{lemma}[\cite{bob}]\label{omegaminuslemma}
We have that $\omega_+ = ((\omega^{-1})^+)^{-1}$ and $\omega^- =
((\omega^{-1})_-)^{-1}$.
\end{lemma}
Hence, if we have a combinatorial description of $\omega^+$ and
$\omega_-$, then we also have a combinatorial description of all
homotopy strings shown in Figure \ref{omeganeighbours}.  The
combinatorial descriptions of $\omega^+$ and $\omega_-$ are given in
Tables \ref{omegaplus}--\ref{omegaminustrivial}.  In Tables
\ref{omegaplus} and \ref{omegaplustrivial} we include the integer
$m'(\omega)$.  The integer $m''(\omega)$ is equal to $m'(\omega)$ if
$\omega^+$ is non-empty, and otherwise equal to $m'(\omega_+)$.

\begin{proposition}\label{tabellproposisjon}
  Let $\omega[m]$ be a string complex. The middle term
  $\omega^+[m+m'(\omega)]$ in the AR-triangle starting in $\omega$ is
  given by the entries in Tables \ref{omegaplus} and
  \ref{omegaplustrivial}. The middle term $\omega_-[m-m'(\omega_-)]$
  in the AR-triangle ending in $\omega$ is given by the entries in
  Tables \ref{omegaminus} and \ref{omegaminustrivial}.
\end{proposition}

\begin{proof}
See Appendix \ref{bobinskiappendix}.
\end{proof}

\renewcommand{\arraystretch}{1.2}
\begin{longtable}[c]{|p{0.05\textwidth}l|p{0.1\textwidth}|p{0.24\textwidth}|m{0.2\textwidth}|}
\caption{$\omega^+$ for a non-trivial and non-empty homotopy string
  $\omega$\label{omegaplus}}\\
\hline
$\alpha_l(\omega)$ & & condition & $\omega^+$ & $m'(\omega)$ \\
\hline
\endfirsthead
\multicolumn{5}{l}
{\tablename\ \thetable\ -- \textit{Continued from previous page}} \\
\hline
$\alpha_l(\omega)$ & & condition & $\omega^+$ & $m'(\omega)$ \\
\hline
\endhead
\hline \multicolumn{5}{l}{\textit{Continued on next page}} \\
\endfoot
\endlastfoot
$\alpha_i$, & $1 \leq i \leq r_2$ & & $cw\_r(t\omega)\cdot \omega$ &
$-1$ \\ \hline \pagebreak[3]
$\alpha_i$, & $r_2 + 1 \leq i \leq r$ & $l(\omega) = 1$ & $\emptyset$ &
-- \\ \hline \pagebreak[3]
$\alpha_i$, & $r_2 + 1 \leq i \leq r$ & $l(\omega) > 1$ & ccw
r-reduction of $\omega$ & 0 \\ \hline \pagebreak[3]
$\beta_i$, & $1 \leq i \leq s_2$ & & $ccw\_s(t\omega)\cdot \omega$ &
$-1$ \\ \hline \pagebreak[3]
$\beta_i$, & $s_2 + 1 \leq i \leq s$ & $l(\omega) = 1$ & $\emptyset$ &
-- \\ \hline \pagebreak[3]
$\beta_i$, & $s_2 + 1 \leq i \leq s$ & $l(\omega) > 1$ & cw
s-reduction of $\omega$ & 0 \\ \hline \pagebreak[3]
$\alpha_i^{-1}$, & $1 \leq i \leq r$ & & $cw\_r(t\omega)\cdot \omega$ & 
\parbox[t]{\columnwidth}{$-1$ if $2 \leq i \leq r_2 +1$ \\
  $0$ if $r_2 + 2 \leq i \leq r$\\
  $\phi(r_1)$ if $i=1$
}
\\ \hline \pagebreak[3]
$\beta_i^{-1}$, & $1 \leq i \leq s$ & & $ccw\_s(t\omega)\cdot \omega$ & $0$ or $-1^*$
\\ \hline \pagebreak[3]
$\gamma_{2i}$, & $1 \leq i \leq r_2$ & & $cw\_r(t\omega)\cdot \omega$
& \parbox[t]{\columnwidth}{
  $0$ if $i=1$ and $r_1\! >\! 0$\! \\
  $-1$ otherwise \vspace{1pt}
  }\\ \hline \pagebreak[3]
$\delta_{2i}$, & $1 \leq i \leq s_2$ & & $ccw\_s(t\omega)\cdot \omega$
& $0$ or $-1^*$ \\ \hline \pagebreak[3]
$\gamma_{2i-1}$, & $1 \leq i \leq r_2$ & & $ccw\_s(t\omega)\cdot
\omega$ & \parbox[t]{\columnwidth}{$\phi(s_1)$ if $i>1$\\
  $\phi(r_1)$ if $i\! =\! 1$ and $s_2\! >\! 0$ \\
  $0$ otherwise \vspace{2pt}} \\ \hline \pagebreak[3]
$\delta_{2i-1}$, & $1 \leq i \leq s_2$ & & $cw\_r(t\omega)\cdot
\omega$ & $0$ or $-1^*$ \\ \hline \pagebreak[3]
$\gamma_{2i}^{-1}$, & $1 \leq i \leq r_2$ & & $ccw\_s(t\omega)\cdot
\omega$ & $\phi(s_1)$ \\ \hline \pagebreak[3]
$\delta_{2i}^{-1}$, & $1 \leq i \leq s_2$ & & $cw\_r(t\omega)\cdot
\omega$ & $0$ or $-1^*$ \\ \hline \pagebreak[3]
$\gamma_{2i-1}^{-1}$, & $1 \leq i \leq r_2$ & & ccw r-reduction of
$\omega$ & $-1$ \\ \hline \pagebreak[3]
$\delta_{2i-1}^{-1}$, & $1 \leq i \leq s_2$ & & cw s-reduction of
$\omega$ & $-1$ \\ \hline \pagebreak[3]
\multicolumn{5}{p{\linewidth}}{\parbox[t]{\linewidth}{\small ${}^{*}$ as in above row, but
  interchanging $r$ and $s$.\\ $\phi: \mathbb{Z}_{\geq 0} \rightarrow
  \{-1,0\}$ is defined by $\phi(a) = -1$ if $a=0$, otherwise $\phi(a) = 0$. }} \\ 
\end{longtable}

\begin{longtable}[c]{|ll|l|l|}
  \caption{$\omega_-$ for a non-trivial and non-empty homotopy string
    $\omega$\label{omegaminus}}\\
  \hline
  $\alpha_l(\omega)$ & & conditions & $\omega_-$ \\
  \hline
  \endfirsthead
  \multicolumn{4}{l}
  {\tablename\ \thetable\ -- \textit{Continued from previous page}} \\
  \hline
  $\alpha_l(\omega)$ & & conditions & $\omega_-$ \\
  \endhead
  \hline \multicolumn{4}{l}{\textit{Continued on next page}} \\
  \endfoot
  \endlastfoot
  $\alpha_i$, & $1 \leq i \leq r$ & & $ccw\_r(t\omega)\cdot \omega$
  \\
  \hline
  $\beta_i$, & $1 \leq i \leq s$ & & $cw\_s(t\omega)\cdot \omega$ \\
  \hline
  $\alpha_i^{-1}$, & $1 \leq i \leq r_2$ & & $ccw\_r(t\omega)\cdot
  \omega$ \\
  \hline
  $\alpha_i^{-1}$, & $r_2 + 1 \leq i \leq r$ & $l(\omega) = 1$ &
  $\emptyset$ \\
  \hline
  $\alpha_i^{-1}$, & $r_2 + 1 \leq i \leq r$ & $l(\omega) > 1$ & cw
  r-reduction of $\omega$ \\
  \hline
  $\beta_i^{-1}$, & $1 \leq i \leq s_2$ & & $cw\_s(t\omega)\cdot
  \omega$ \\
  \hline
  $\beta_i^{-1}$, & $s_2 + 1 \leq i \leq s$ & $l(\omega) = 1$ &
  $\emptyset$ \\
  \hline
  $\beta_i^{-1}$, & $s_2 + 1 \leq i \leq s$ & $l(\omega) > 1$ & ccw
  s-reduction of $\omega$ \\
  \hline
  $\gamma_{2i}$, & $1 \leq i \leq r_2$ & & cw r-reduction of $\omega$
  \\
  \hline
  $\gamma_{2i-1}$, & $1 \leq i \leq r_2$ & & $cw\_s(t\omega)\cdot
  \omega$ \\
  \hline
  $\delta_{2i}$, & $1 \leq i \leq s_2$ & & ccw s-reduction of $\omega$
  \\
  \hline
  $\delta_{2i-1}$, & $1 \leq i \leq s_2$ & & $ccw\_r(t\omega)\cdot
  \omega$ \\
  \hline
  $\gamma_{2i}^{-1}$, & $1 \leq i \leq r_2$ & & $cw\_s(t\omega)\cdot
  \omega$ \\
  \hline
  $\gamma_{2i-1}^{-1}$, & $1 \leq i \leq r_2$ & & $ccw\_r(t\omega)
  \cdot \omega$ \\
  \hline
  $\delta_{2i}^{-1}$, & $1 \leq i \leq s_2$ & & $ccw\_r(t\omega)\cdot
  \omega$ \\
  \hline
  $\delta_{2i-1}^{-1}$, & $1 \leq i \leq s_2$ & & $cw\_s(t\omega)
  \cdot \omega$ \\
  \hline
\end{longtable}

%
%
\renewcommand{\arraystretch}{1.05}
\begin{longtable}[c]{|>{\centering\arraybackslash}m{0.14\linewidth}|m{0.20\linewidth}|m{0.18\linewidth}|m{0.22\linewidth}|m{0.07\linewidth}|}
\caption{\small $\omega^+$ for a trivial homotopy string\label{omegaplustrivial}}\\
\hline
$x$ in & $\omega^+$ for $\omega=1_{x}$ & $m'(\omega)$ &$\omega^+$ for
$\omega = 1_{x}^{-1}$ & $m'(\omega)$ \\
\hline
\endfirsthead
\multicolumn{5}{l}%
{\tablename\ \thetable\ -- \textit{Continued from previous page}} \\
\hline
$x$ in & $\omega^+$ for $\omega=1_{x}$ & $m'(\omega)$ & $\omega^+$ for
$\omega = 1_{x}^{-1}$ & $m'(\omega)$ \\
\hline
\endhead
\hline \multicolumn{5}{l}{\textit{Continued on next page}} \\
\endfoot
\endlastfoot
$A$ & \multirow{2}{*}{$cw\_r(x)$} &$\phi(r_1)$ if $i\!=\!1$& \multirow{2}{*}{$\emptyset$} & \multirow{2}{*}{--}\\
(so $x=A_i$) & &$0$ otherwise & & \\
\hline\pagebreak[3]
$A'$ & \multirow{2}{*}{$ccw\_s(x)$} & $\phi(s_1)$ if $i\!=\!1$ &  \multirow{2}{*}{$\emptyset$} & \multirow{2}{*}{--}\\
(so $x=A'_i$) & &$0$ otherwise & & \\
\hline\pagebreak[3]
$B$ & $cw\_r(x)$ & $-1$ & $ccw\_s(x)$ & $\phi(s_1)$ \\
\hline\pagebreak[3]
$B'$ & $ccw\_s(x)$ &$-1$ & $cw\_r(x)$ & $\phi(r_1)$\\
\hline\pagebreak[3] \multirow{2}{*}{$C$} & $1_{D_{r_1-1}} $ if $r_1 >
0$ & $0$ & $1_{D'_{s_1-1}}$ if $s_1 > 0$
&$0$ \\
& $cw\_r(x) $ if $r_1 = 0$ & $-1$ & $ccw\_s(x)$ if $s_1 = 0$ & $-1$ \\
\hline\pagebreak[3]
$D$ & $1_{D_{i-1}} $ if $i > 0$ & $0$ & \multirow{2}{*}{$ccw\_s(x)$} &  \multirow{2}{*}{$\phi(s_1)$} \\
(so $x = D_i$) & $cw\_r(x)$ if $i = 0$  &$-1$ &  & \\
\hline\pagebreak[3]
\multirow{2}{*}{\parbox[c][1.3cm][c]{\linewidth}{
    \centering $D'$\\(so $x = D'_i$)}}
& $1_{D'_{i-1}} $ if $i > 0$ & $0$& \multirow{3}{*}{$cw\_r(x)$} &  \multirow{3}{*}{$\phi(r_1)$}\\
 & $ccw\_s(x)$ if $i = 0$ &
$\!\left\{\! \begin{array}{l} 0 \text{ if } s_2\!=\!0 \\
  -1 \text{ otherwise} \end{array} \right. $ & & \\
\hline\pagebreak[3]
$F$ & $ccw\_s(x)$ & $\phi(s_1)$ & $cw\_r(x)$&  $\phi(r_1)$\\
\hline
$F'$ & $cw\_s(x)$ & $-1$ & $ccw\_r(x)$& $-1$ \\
\hline \multicolumn{5}{p{0.85\linewidth}}{\small $\phi: \mathbb{Z}_{\geq 0} \rightarrow
  \{-1,0\}$ is defined by $\phi(a) = -1$ if $a=0$, otherwise $\phi(a) = 0$.} \\
\end{longtable}
\begin{longtable}[c]{|c|p{0.35\linewidth}|p{0.20\linewidth}|}
\caption{\small $\omega_-$ for a trivial homotopy string\label{omegaminustrivial}}\\
\hline
$x$ in & $\omega_-$ for $\omega = 1_{x}$ & $\omega_-$ for $\omega = 1_{x}^{-1}$ \\
\hline
\endfirsthead
\multicolumn{3}{l}%
{\tablename\ \thetable\ -- \textit{Continued from previous page}} \\
\hline
$x$ in & $\omega_-$ for $\omega = 1_{x}$ & $\omega_-$ for $\omega = 1_{x}^{-1}$ \\
\hline
\endhead
\hline \multicolumn{3}{l}{\textit{Continued on next page}} \\
\endfoot
\endlastfoot
$A$ & $\emptyset$ & $ccw\_r(x)$\\
\hline\pagebreak[3]
$A'$ & $\emptyset$ & $cw\_s(x)$\\
\hline\pagebreak[3]
$B$ & $ccw\_r\_p(x)$ & $cw\_s\_p(x)$\\
\hline\pagebreak[3]
$B'$ & $cw\_s\_p(x)$ & $ccw\_r\_p(x)$\\
\hline\pagebreak[3]
$C$ & $ccw\_r(x)$ & $cw\_s(x)$\\
\hline\pagebreak[3]
$D_i$ & $1_{D_{i+1}}$ if $i < r_1-1$ & \multirow{2}{*}{$cw\_s(x)$}\\
(so $x = D_i$) & $1_C$ if $i = r_1 - 1$ and $s_1 > 0$ &  \\
\hline\pagebreak[3]
$D'_i$ & $1_{D'_{i+1}}$ if $i < s_1-1$ & \multirow{2}{*}{$ccw\_r(x)$}\\ 
(so $x = D'_i$) & $1_C^{-1}$ if $i = s_1 - 1$ and $r_1 > 0$ &  \\
\hline\pagebreak[3]
$F$ & $cw\_r\_p(x)$ & $ccw\_s\_p(x)$\\
\hline
$F'$ & $ccw\_r\_p(x)$ & $cw\_s\_p(x)$\\
\hline
\end{longtable}
For the remaining part of this section, we will only consider
properties of homotopy strings, and not complexes in $\kbp$.  Note that
it will never happen that $\omega$ is equal to any of $\omega^+$,
$\omega_+$, $\omega^-$ and $\omega_-$.

We start by defining four diagonals for a given homotopy string
$\omega$. The \emph{upper right diagonal} of $\omega$ is the
sequence $(\omega, \omega^+, (\omega^+)^+, \ldots)$. Furthermore, we
define the \emph{lower right diagonal} of $\omega$ to be the
sequence $(\omega, \omega_+, (\omega_+)_+, \ldots)$. The upper and
lower left diagonals of $\omega$ are defined similarly.

Next, let a \emph{$Q$-walk} denote one of the following walks:
clockwise r-walk, clockwise s-walk, counter-clockwise r-walk,
counter-clockwise s-walk. Note that it follows from Proposition
\ref{tabellproposisjon}, that the condition $l(\omega^+)>l(\omega)$
implies that $\omega^+$ is of the form $\sigma\omega$ for a step
$\sigma$ in a $Q$-walk.
\begin{corollary}\label{diagonaler}
  Let $\omega$ be a homotopy string, let $\omega^*$ be either
  $\omega^+$ or $\omega_-$, and let $\mathcal{D}$ be the diagonal
  $(\omega, \omega^*, \ldots)$. If $\omega^*$ is of
  the form $\sigma \omega$, where $\sigma$ is a step in a
  $Q$-walk $W$, then the homotopy string in $\mathcal D$ succeeding
  $\omega^*$ is $\sigma' \omega^*$, where $\sigma'$ is the step
  succeeding $\sigma$ in $W$.
\end{corollary}

\begin{proof}
  This follows from the definition of walks and Tables
  \ref{omegaplus}--\ref{omegaminustrivial}.  For example, assume that
  $\omega^{*} = \omega^+$ and that $\omega^+ = \sigma \omega$ where
  $\sigma$ is a step in a clockwise r-walk. Then $t(\omega^{+}) \in A
  \cup D$, such that $\alpha_l(t(\omega^+))$ is an arrow $\gamma_{2i}$
  if $t(\omega^+) \in A$, and an inverse r-arrow otherwise.  In all
  these cases, $(\omega^+)^+ = cw\_r(t(\omega^+))\cdot \omega^+$.
\end{proof}

\begin{proposition}\label{todiagonaler}
  (1) Let $\omega$ be a homotopy string. Assume that
  $\omega^{+}=\sigma\omega$ and $\omega_+=\omega\sigma'$, where
  $\sigma$ is a step in a $Q$-walk $W$ and $\sigma'^{-1}$ is a step in a
  $Q$-walk $W'$. If $\widetilde{\omega}$ is in the lower right
  diagonal of $\omega$, and  $\widehat{\omega}$ is in the upper
  right diagonal of $\omega$, then 
  \[\widetilde{\omega}^+ = \sigma\widetilde{\omega} ~~\text{ and }~~ \widehat{\omega}_+ =
  \widehat{\omega}\sigma' ~.\]
  (2) Let $\omega$ be a homotopy string. Assume that
  $\omega^{+}=\sigma\omega$ and $\omega^-=\omega\sigma'$, where
  $\sigma$ is a step in a $Q$-walk $W$ and $\sigma'^{-1}$ is a step in a
  $Q$-walk $W'$. If $\widetilde{\omega}$ is in the upper left
  diagonal of $\omega$, and $\widehat{\omega}$ is in the upper
  right diagonal of $\omega$, then
  \[\widetilde{\omega}^+ = \sigma\widetilde{\omega} ~~\text{ and }~~ \widehat{\omega}^- =
  \widehat{\omega}\sigma' ~.\]
  (3) Let $\omega$ be a homotopy string. Assume that
  $\omega_{-}=\sigma\omega$ and $\omega^-=\omega\sigma'$, where
  $\sigma$ is a step in a $Q$-walk $W$ and $\sigma'^{-1}$ is a step in a
  $Q$-walk $W'$. If $\widetilde{\omega}$ is in the upper left
  diagonal of $\omega$, and $\widehat{\omega}$ is in the lower
  left diagonal of $\omega$, then
  \[\widetilde{\omega}_- = \sigma\widetilde{\omega} ~~\text{ and }~~ \widehat{\omega}^- =
  \widehat{\omega}\sigma' ~.\] (4) Let $\omega$ be a homotopy
  string. Assume that $\omega_-=\sigma\omega$ and
  $\omega_+=\omega\sigma'$, where $\sigma$ is a step in a $Q$-walk $W$
  and $\sigma'^{-1}$ is a step in a $Q$-walk $W'$. If
  $\widetilde{\omega}$ is in the lower left diagonal of $\omega$, and
  $\widehat{\omega}$ is in the lower right diagonal of $\omega$, then
  \[\widetilde{\omega}_+ = \widetilde{\omega}\sigma' ~~\text{ and }~~ \widehat{\omega}_- =
  \sigma\widehat{\omega} ~.\]
\end{proposition}
\begin{proof}
  We prove case (1).  The proofs for cases (2)--(4) are similar. If
  $\omega \neq \emptyset$ is non-trivial, then by Corollary
  \ref{diagonaler}, we have that $\alpha_1(\widehat{\omega}) =
  \alpha_1(\omega)$ and that
  $\alpha_{l(\widetilde{\omega})}(\widetilde{\omega}) =
  \alpha_{l(\omega)}(\omega)$; and $\widehat{\omega}_+$ is determined
  by $\alpha_1(\widehat{\omega})$, and $\widetilde{\omega}^+$ is
  determined by $\alpha_{l(\widetilde{\omega})}$.

  Assume now that $\omega$ is trivial, that is, $\omega =
  1_x^{\varepsilon}$ for some vertex $x$ and some $\varepsilon \in
  \{-1,1\}$. Then, by Table \ref{omegaplustrivial}, it follows that
  $x$ is in $B \cup B' \cup D_0 \cup D_0' \cup F \cup F'$, since these are the
  only vertex types where both $1_x^+$ and $(1_x^{-1})^+$ are of
  length greater than $\omega$. It is easy to verify by Tables
  \ref{omegaplus} and \ref{omegaplustrivial} that we are in this
  situation:
  
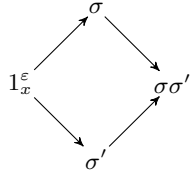
\[
  \begin{tikzpicture}
    \node (1) at (-1,0) [tvertex] {$1_x^{\varepsilon}$};
    \node (2) at (0,1) [tvertex] {$\sigma$};
    \node (3) at (0,-1) [tvertex] {$\sigma'$};
    \node (4) at (1,0) [tvertex] {$\sigma \sigma'$};
    \draw [->] (1)--(2);
    \draw [->] (1)--(3);
    \draw [->] (2)--(4);
    \draw [->] (3)--(4);
  \end{tikzpicture}
  \]
  We can now consider the upper right diagonal of  $\sigma$ and the
  lower right diagonal of $\sigma'$, but then we are in the
  non-trivial case.
\end{proof}

\section{Classification of the AR-components}
In this section, we give a complete classification of all the
AR-components in $\kbp$ where $\Lambda \cong kQ/I$ is a fixed
cluster-tilted algebra of type $\ant$ with $Q$ in $\mathcal
Q_n^{\star}$. 

We start by defining admissible reduction for homotopy strings and we
show that there are three types of homotopy strings that can not be
admissibly reduced. Moreover, any other homotopy string can be
admissibly reduced to a homotopy string of one of these three types.
In Section \ref{sec:charcomp} we show that one of the classes of
homotopy strings that can not be admissibly reduced, gives rise to the
$\zainf$-components and tubes containing string complexes.  In Section
\ref{sec:zainfinf} we parametrize the $\zainfinf$-components arising
from the second class of homotopy strings that can not be admissibly
reduced, and we describe the $\zainfinf$-components arising from the
third class of homotopy strings that can not be admissibly reduced.
In Section \ref{sec:band} we consider the AR-components containing
band complexes.  Finally, Section \ref{summary} provides a summary of
the main results.

In the remainder of this section, all homotopy strings and homotopy
bands are associated with the fixed algebra $\Lambda$.

\begin{definition}\label{admissible}
Let $\omega$ be a homotopy string which is neither an r-arrow, nor an
s-arrow, nor the inverse of such an arrow. If $\omega'$ is either the
clockwise r-reduction of $\omega$, or the counter-clockwise
r-reduction of $\omega$, or the clockwise s-reduction of $\omega$, or
the counter-clockwise s-reduction of $\omega$, then we call $\omega'$
a \emph{left admissible reduction} of $\omega$.

We define a \emph{right admissible reduction} of $\omega$ to be $\omega''$,
where $(\omega'')^{-1}$ is a left admissible reduction of
$\omega^{-1}$.
\end{definition}

If there exists a left or right admissible reduction of a homotopy
string $\omega$, then we say that $\omega$ can be admissibly reduced,
and the operation performed is an admissible reduction.

\begin{lemma}\label{admredlemma}
  Let $\omega$ be a homotopy string. If $\omega'$ is a left admissible
  reduction of $\omega$, then either $\omega'=\omega^+$ or
  $\omega'=\omega_-$. Moreover,
\begin{itemize}
\item if $\omega'$ is a clockwise r-reduction or a counter-clockwise
  s-reduction of $\omega$, then $\omega'=\omega_-$, and
\item if $\omega'$ is a clockwise s-reduction or a counter-clockwise
  r-reduction of $\omega$, then $\omega'=\omega^+$.
\end{itemize}
Similarly, if $\omega'$ is a right admissible reduction of $\omega$,
then $\omega'$ is either ${\omega_+}$ or $\omega^-$.
\end{lemma}
\begin{proof}
  Let $\omega$ be a homotopy string and assume it can be left
  admissibly reduced and that $\omega'$ is the left admissible
  reduction of $\omega$.  The left admissible reduction of $\omega$ is
  either a clockwise r-reduction of $\omega$, a counter-clockwise
  r-reduction of $\omega$, a clockwise s-reduction of $\omega$ or a
  counter-clockwise s-reduction of $\omega$.

  If the reduction is a clockwise r-reduction, then $\omega$ satisfies
  condition (\ref{condition1}) in chapter \ref{redsubkap}. In
  particular $\alpha_l(\omega)$ is either an arrow $\gamma_{2i}$ for
  some $1 \leq i \leq r_2$, or an inverse r-arrow. Thus we need to
  check that in all these cases $\omega'=\omega_-$. The entries in
  Table \ref{omegaminus} clearly show that this is so, except for the
  case when $l(\omega)=1$ and $\omega$ is the inverse of an
  r-arrow. The clockwise r-reduction of the exceptions are defined,
  however the reductions are not admissible reductions, and therefore
  are not cases we need to consider, as we have assumed that $\omega$ can
  be admissibly reduced.
  
  The three other cases follow by similar arguments.
\end{proof}

\begin{corollary}\label{admarcorollar}
  Let $\omega$ be a homotopy string. Then $\omega'$ is a left
  admissible reduction of $\omega$ if and only if $\omega'\neq
  \emptyset$ and $\omega'$ is one of $\{\omega^+, \omega_-\}$ such
  that $l(\omega') < l(\omega)$.
\end{corollary}
\begin{proof}
Let $\omega$ be a homotopy string.

Assume that $\omega$ has an admissible reduction $\omega'$. By the
definition of admissible reduction it is clear that
$l(\omega')<l(\omega)$ and $\omega' \neq \emptyset$, and by Lemma
\ref{admredlemma} either $\omega'=\omega^+$ or $\omega'=\omega_-$.

Assume that $\omega' \neq\emptyset$ is either $\omega^+$ or $\omega_-$
and that $l(\omega')<l(\omega)$. By Tables \ref{omegaplus} and
\ref{omegaminus} it is clear that in any of these cases $\omega^+$ and
$\omega_-$ is an admissible reduction.
\end{proof}
It is clear from the definition of a right admissible reduction, that
there is a similar result as Corollary \ref{admarcorollar} for right
admissible reductions.  We now give the definition of central homotopy
strings, which we will show form one of the classes of homotopy strings
that can not be admissibly reduced.

\begin{definition}
  A \emph{central homotopy string} is either a trivial homotopy string
  corresponding to a vertex of type $B, B', F$ or $F'$, or a homotopy
  string $\omega$ where
  \begin{itemize}
  \item $\alpha_1(\omega)\in\left\{\alpha_i\,, \alpha_i^{-1}\,, \beta_j\,,
      \beta_j^{-1}\,, \gamma_{2i}\,, \gamma_{2i-1}^{-1}\,, \delta_{2j}\,,
      \delta_{2j-1}^{-1}~|\,1\leq i\leq r_2, 1\leq j\leq s_2 \right\}$
  \item $\alpha_l(\omega)\in\left\{\alpha_i\,, \alpha_i^{-1}\,, \beta_j\,,
      \beta_j^{-1}\,, \gamma_{2i-1}\,, \gamma_{2i}^{-1}\,, \delta_{2j-1}\,,
      \delta_{2j}^{-1}~|\,1\leq i\leq r_2, 1\leq j\leq s_2 \right\}.$
  \end{itemize}
\end{definition}

\begin{lemma}\label{centralnoreduction}
  If $\omega$ is a central homotopy string, then $\omega$ can not be
  admissibly reduced.
\end{lemma}
\begin{proof}
  This follows from Tables \ref{omegaplus} and \ref{omegaminus}.
\end{proof}

We have now seen that there are three classes of homotopy strings
which can not be admissibly reduced: The central homotopy strings (by
the above lemma), the non-central trivial homotopy strings (because no
trivial homotopy string can be reduced, by Definition
\ref{reduction1}), and the r- and s-arrows and their inverses (by
Definition \ref{admissible}).  The following lemma shows that these
classes of homotopy strings are the only ones which can not be
admissibly reduced.

\begin{lemma}\label{reduksjonstyper}
  Let $\omega \neq \emptyset$ be a homotopy string.  By a series of
  left or right admissible reductions, $\omega$ can be reduced to a
  homotopy string which is of one of the following types:
  \begin{itemize}
  \item[(i)] an r- or s-arrow or an inverse of such an arrow or a
    trivial homotopy string corresponding to a vertex of type $A$ or
    $A'$, or
    \item[(ii)] a central homotopy string, or
    \item[(iii)] a trivial homotopy string corresponding to a vertex
      of type $C$, $D$ or $D'$.
  \end{itemize}
\end{lemma}

\begin{proof}
  Let $\omega \neq \emptyset$ be a homotopy string, and assume that
  $\omega$ is neither of type $(i)$, nor $(ii)$, nor $(iii)$.  Then
  $l(\omega) > 0$. Denote by $X$ the complement of
  \[\left\{\alpha_i\,, \alpha_i^{-1}\,, \beta_j\,, \beta_j^{-1}\,,
    \gamma_{2i-1}\,, \gamma_{2i}^{-1}\,, \delta_{2j-1}\,,
    \delta_{2j}^{-1}~|\,1\leq i\leq r_2, 1\leq j\leq s_2 \right\}\] in
  $Q'_1$. Let $\widehat{\omega} = \omega$ if $\alpha_l(\omega) \in
  X$, otherwise let $\widehat{\omega} = \omega^{-1}$.  It is clear
  that $\alpha_l(\widehat{\omega}) \in X$, since $\omega$ is not a
  central homotopy string.  In any case, either $\widehat{\omega}^+$
  or $\widehat{\omega}_-$ will be an admissible reduction of
  $\widehat{\omega}$.

  For instance, assume that $\alpha_l(\widehat{\omega}) = \gamma_{2i}$
  for some $1\leq i \leq r_2$. Then by Table \ref{omegaminus},
  $\widehat{\omega}_-$ is an admissible reduction of
  $\widehat{\omega}$.

  Next, let $\widehat{\widehat{\omega}}$ be the admissible reduction
  of $\widehat{\omega}$.  Repeat the above step for the homotopy
  string $\widehat{\widehat{\omega}}$.
\end{proof}

\subsection{Characteristic components containing string complexes}
\label{sec:charcomp}
By a characteristic component, we mean an AR-component containing
AR-triangles with only one middle term.  In this section, we will
consider characteristic components containing string complexes.
These components are dependent on the parameters of the quiver of the
cluster-tilted algebra.

A similar result as Proposition \ref{randproposisjon} holds by
exchanging the parameters $s_1$ and $s_2$ with the parameters $r_1$
and $r_2$.

\begin{proposition}\label{randproposisjon}
  If $s_2 = 0$ then, for each $i \in \mathbb Z$, there is a characteristic
  component with the following edge:

\[
\begin{tikzpicture}
  \node (1) at (0,0) [kvertex] {$\beta_s^{-1}[i]$};
  \node (2) at (2,0) [kvertex] {$\beta_{s-1}^{-1}[i]$};
  \node (3) at (3.25,0) [kvertex] {$\cdots$};
  \node (4) at (4.5,0) [kvertex] {$\beta_1^{-1}[i]$};
  \node (5) at (6.5,0) [kvertex] {$\beta_s^{-1}[i]$};
  \node (9) at (3.25,1) [kvertex] {$\cdots$};
  \draw [->] (1)--(0.9,1);
  \draw [->] (1.1,1)--(2);
  \draw [->] (2) -- (2.9,1);
  \draw [->] (3.6,1)--(4);
  \draw [->] (4)--(5.4,1);
  \draw [->] (5.6,1)--(5);
  \draw [dotted] (0,0.3)--(0,1);
  \draw [dotted] (6.5,0.3)--(6.5,1);
\end{tikzpicture}
\]

If $s_2 \neq 0$, there is a class of $s_2$ AR-components with the
following edges:

\[
\begin{tikzpicture}
  \node (0) at (-0.5,0.5) [kvertex] {$\cdots$};
  \node (1) at (0,0) [kvertex] {$1_{A'_{s_2}}$};
  \node (s1) at (0,-0.4) [kvertex] {$[i]$};
  \node (2) at (2,0) [kvertex] {$1_{A'_{s_2-1}}$};
  \node (s2) at (2,-0.4) [kvertex] {$[i-1]$};
  \node (3) at (3.25,0) [kvertex] {$\cdots$};
  \node (4) at (4.5,0) [kvertex] {$1_{A'_1}$};
  \node (s4) at (4.5,-0.4) [kvertex] {$[i-s_2+1]$};
  \node (5) at (6.5,0) [kvertex] {$\beta_s^{-1}$};
  \node (s5) at (6.5,-0.4) [kvertex] {$[i-s_2+1]$};
  \node (7) at (7.75,0) [kvertex] {$\cdots$};
  \node (6) at (9,0) [kvertex] {$\beta_1^{-1}$};
  \node (s6) at (9,-0.4) [kvertex] {$[i-s_2+1]$};
  \node (8) at (11,0) [kvertex] {$1_{A'_{s_2}}$};
  \node (s8) at (11,-0.4) [kvertex] {$[i-s_2]$};
  \node (9) at (3.25,1) [kvertex] {$\cdots$};
  \node (10) at (7.75,1) [kvertex] {$\cdots$};
  \node (11) at (11.25, 0.5) [kvertex] {$\cdots$};
  \draw [->] (1)--(0.9,1);
  \draw [->] (1.1,1)--(2);
  \draw [->] (2) -- (2.9,1);
  \draw [->] (3.6,1)--(4);
  \draw [->] (4)--(5.4,1);
  \draw [->] (5.6,1)--(5);
  \draw [->] (5)--(7.4,1);
  \draw [->] (8.1,1)--(6);
  \draw [->] (6)--(9.9,1);
  \draw [->] (10.1,1)--(8);
\end{tikzpicture}
\]

\end{proposition}

\begin{proof}
  From Proposition \ref{tabellproposisjon}, it is clear that for any
  of the homotopy strings $1_{A'_i}$ for $1 \leq i \leq s_2$ and
  $\beta_j^{-1}$ for $s_2+1 \leq j \leq s$, we have
  $\omega_+=\emptyset$. Hence the string complexes shown in the above
  figures are all in some characteristic component.  The rest of the
  result follows from direct calculations.
\end{proof}

The characteristic components described in Proposition
\ref{randproposisjon} are called s-components.  Similarly, the
characteristic components depending on the parameters $r_1$ and $r_2$
are called r-components.

In the next corollaries we show that the r- and s-components are
tubes or $\zainf$-components, and that they are exactly the
characteristic AR-components containing string complexes.  We also
describe the string complexes occurring in such components. 

\begin{corollary}\label{allcharacteristicstrings}
  Let $\omega[i]$ be a string complex occurring in an s-component.
  Then $\omega$ is of the following form: $\omega = w_k\cdots w_1
  \omega'$ where $\omega'[j]$ is on the edge of the component (for
  some $j\in \mathbb{Z}$), and where $w_1, \cdots, w_k$ are the $k$
  first consecutive steps of the counter-clockwise s-walk starting
  in $t\omega'$.

  Similarly, let $\omega[i]$ be a string complex occurring in an
  r-component.  Then $\omega$ is of the form $\omega = w_k \cdots
  w_1 \omega'$ where $\omega'[j]$ is on the edge of the component (for
  some $j \in \mathbb{Z}$), and where $w_1, \cdots, w_k$ are the $k$
  first consecutive steps of the clockwise r-walk starting in
  $t\omega'$.
\end{corollary}

\begin{proof}
  From Tables \ref{omegaplus}--\ref{omegaminustrivial}, we know that
  for a string complex $\omega[i]$ on the edge of an s-component, we
  have that $\omega^+ = ccw\_s(t\omega)\cdot \omega$. Then, from
  Corollary \ref{diagonaler}, we know that the upper right diagonal of
  $\omega$ will continue to grow with succeeding steps in a 
  counter-clockwise s-walk.

  The other case is proved similarly.
\end{proof}

\begin{corollary}
  If $s_2=0$, then the s-components are tubes of rank $s_1$.  If
  $s_2 >0$, then the s-components are of type $\zainf$ with $\tau^s
  = [s_2]$.
\end{corollary}

\begin{corollary}\label{allcomp}
  The r- and s-components are exactly the characteristic
  components containing string complexes.
\end{corollary}
\begin{proof}
  This is clear, as in Tables
  \ref{omegaplus}--\ref{omegaminustrivial} the empty homotopy string
  $\emptyset$ occurs only for the homotopy strings listed in
  Proposition \ref{randproposisjon}.
\end{proof}

\begin{corollary}\label{nocentrals}
Let $\omega[i]$ be a string complex in a characteristic
component. Then $\omega$ is not a central homotopy string.
\end{corollary}
\begin{proof} If $\omega[i]$ is a string complex on the edge of a
  characteristic component, then $\omega$ is not a central homotopy
  string.  Let $\omega'[j]$ be a string complex in a characteristic
  component, which is not on the edge.  Then by Corollary
  \ref{allcharacteristicstrings} it is clear that $\omega'$ can be
  admissibly reduced.  By Lemma \ref{centralnoreduction} no central
  homotopy string can be in this characteristic component.
\end{proof}

From Corollary \ref{allcharacteristicstrings}, it is clear that no
string complex $\omega[i]$ where $\omega$ is of type (iii) in Lemma
\ref{reduksjonstyper} can be in a characteristic component.

\subsection{The $\zainfinf$-components}
\label{sec:zainfinf}

From Corollary \ref{allcomp}, Corollary
\ref{nocentrals} and Section \ref{sec:components}, it is clear that if
$\omega[i]$ is a string complex where $\omega$ is of type (ii) or
(iii) in Lemma \ref{reduksjonstyper}, then $\omega[i]$ is in a
component of type $\zainfinf/G$ where $G$ is an admissible group of
automorphisms on $\zainfinf$.

Let $\omega[i]$ be a string complex where $\omega$ is a central
homotopy string.  In the following lemma, we show that if $\upsilon$
is another central homotopy string, then the string complex
$\upsilon[j]$ can not be in the same AR-component as $\omega[i]$.

\begin{lemma}\label{lem:centralhomstr}
  Let $\omega[i]$ and $\upsilon[j]$ be string complexes where $\omega$
  and $\upsilon$ are central homotopy strings. If $\omega[i] \ncong
  \upsilon[j]$, then $\omega[i]$ and $\upsilon[j]$ are in different
  AR-components.
\end{lemma}
\begin{proof}
  Let $\omega$ be a central homotopy string. From Tables
  \ref{omegaplus}--\ref{omegaminustrivial}, it is easy to verify that
  each of $\omega^+$, $\omega_+$, $\omega^-$ and $\omega_-$ is
  obtained by adding (possibly the inverse of) a step of some $Q$-walk
  to $\omega$, and that the four $Q$-walks involved are of four
  different types. By Corollary \ref{diagonaler} and Proposition
  \ref{todiagonaler}, there can be no central homotopy string
  different from $\omega$ in the AR-component.
\end{proof}

It follows from the proof of Lemma \ref{lem:centralhomstr} that the
string complex $\omega[i]$, where $\omega$ is a central homotopy
string, never occurs twice in the same component.  Hence, the only
possibility for $G$ is the trivial group, and thus there is a class of
$\zainfinf$-components which up to shift are parametrized by the
central homotopy strings.  Note that for each $\zainfinf$-component in
this class, the corresponding central homotopy string $\omega$ is of
strictly smaller length than any homotopy string $\nu$ where $\nu[j]$
is in the same component for some $j \in \mathbb Z$.

We now consider string complexes of the form $\omega[i]$ where
$\omega$ is of type (iii) in Lemma \ref{reduksjonstyper}, that is a
trivial homotopy string corresponding to a vertex of type $C, D$ or
$D'$.

For a string complex $\omega[i]$ in an AR-component of type
$\zainfinf$, we define the upper right diagonal of $\omega[i]$ to be
the sequence $(\omega[i],\omega^+[i'],
(\omega^+)^+[i''],\ldots)$ where $i' = i+m(\omega)$ and $i'' =
i'+m(\omega^+)$.  Similarly, we define the lower right diagonal of
$\omega[i]$, the upper left diagonal of $\omega[i]$ and the lower left
diagonal of $\omega[i]$.

The following lemma is illustrated in Example 2.1 revisited in
Section \ref{eksempler}.

\begin{lemma}\label{lem:uglycomp}
  All stalk complexes $\omega[i]$ where $\omega$ is a trivial homotopy
  string corresponding to a vertex of type $C, D$ or $D'$ and $i$ is a
  fixed integer, are in the same $\zainfinf$-component.  There are
  irreducible maps, which depend on the values of the parameters $r_1$
  and $s_1$, as described below:
  \begin{itemize}
  \item If $r_1> 0$ the irreducible maps in the lower
    right diagonal of $1_C^{-1}[i]$ are
    \[
    1_C^{-1}[i] \rightarrow 1_{D_{r_1-1}}^{-1}[i] \rightarrow\cdots
    1_{D_1}^{-1}[i]\rightarrow 1_{D_0}^{-1}[i].
    \]
  \item If $s_1 > 0$ the irreducible maps in the upper right diagonal
    of $1_C^{-1}[i]$ are
    \[1_C^{-1}[i]\rightarrow 1_{D'_{r_1-1}}[i]\rightarrow\cdots
    1_{D'_1}[i]\rightarrow 1_{D'_0}[i].
    \]
  \end{itemize}
  Furthermore, for each $j$ in $\mathbb{Z}$ with $j \neq i$ and
  $\omega$ as above, $\omega[j]$ and $\omega[i]$ are not in the same
  component.
\end{lemma}
\begin{proof}
  We have seen that the string complexes arising from trivial homotopy
  strings corresponding to a vertex of type $C, D$ or $D'$ are in a
  component of the form $\zainfinf/G$. By a similar argument as in the
  proof of Lemma \ref{lem:centralhomstr} and its subsequent comment,
  $G$ is trivial.  The irreducible maps follow from Tables
  \ref{omegaplustrivial} and \ref{omegaminustrivial}.
\end{proof}

Hence, if at least one of $r_1$ and $s_1$ is non-zero, we get a class
of $\zainfinf$-components parametrized by $\mathbb Z$.  If $r_1 = s_1
= 0$, then we have no vertices of type $C, D$ or $D'$, and hence no
AR-component of type $\zainfinf$ as in Lemma \ref{lem:uglycomp}.

\subsection{Characteristic components containing band complexes}
\label{sec:band}
For any $Q$ in $\mathcal Q_n^{\star}$ there is always a homotopy band
$\omega=\beta_1^{-1}\cdots\beta_s^{-1}\alpha_r\cdots\alpha_1$, called
\emph{the central homotopy band of $Q$}. Note that it follows from the
definition of a homotopy band that there can be no homotopy bands
starting in a vertex of type $A, A', D_{\geq 1}$ and $D_{\geq 1}'$.  

\begin{proposition}
  Let $\Lambda\cong kQ/I$ be a cluster-tilted algebra of type
  $\widetilde{A}_n$. There are finitely many homotopy bands associated
  with $\Lambda$ if and only $\Lambda$ is hereditary.
\end{proposition}
\begin{proof}
  It is clear that for the hereditary case, the central homotopy band
  is the only homotopy band.  Assume now that $\Lambda$ is not
  hereditary, that is, we have $r_2 > 0$ (see Table
  \ref{parameterverdier}).  We can then construct one class consisting
  of infinitely many homotopy bands for the case when $r=r_2=1$, and
  one class consisting of infinitely many homotopy bands for the case
  when $r>1$.

  First, let $r=r_2=1$.  Then the homotopy string \[\omega =
  \beta_1^{-1}\cdots
  \beta_s^{-1}\gamma_2^{-1}\gamma_1^{-1}\alpha_1^{-1}\beta_s \cdots
  \beta_1\gamma_1\gamma_2\alpha_1\] is a homotopy band. Let 
  $\omega_c$ denote the
  central homotopy band.  Then the homotopy string
$
\omega_n = \omega \cdot (\omega_c)^n
$
is a homotopy band for all positive integers $n$.

Now let $r>1$.  Then the homotopy string \[ \omega = \beta_1^{-1}
\cdots \beta_s^{-1} \alpha_r \cdots \alpha_2
\gamma_2^{-1}\gamma_1^{-1}\alpha_1^{-1}\cdots \alpha_r^{-1}\beta_s
\cdots \beta_1 \gamma_1 \gamma_2 \alpha_1
\]
is a homotopy band. Again, let $\omega_c$ denote the central homotopy
band. Then the homotopy string
$
\omega_n = \omega \cdot (\omega_c)^n
$
is a homotopy band for all positive integers $n$. Thus, it is clear
that any non-hereditary cluster-tilted algebra of type $\ant$ will
have infinitely many homotopy bands associated with itself.
\end{proof}

Note that for most quivers in $\mathcal Q_n^{\star}$, there are also
ways of constructing classes of infinitely many homotopy bands other
than in the proof.

\subsection{Summary}
\label{summary}
In the following theorem we give a full overview of all the
AR-components of $\kbp$.  Note that we always assume that $r_2 > 0$
(see Table \ref{parameterverdier}).
\begin{theorem}\label{finalthm}
  Let $\mathfrak C$ be the set of central homotopy strings associated
  with $\Lambda$, and $\mathfrak B$ the set of homotopy bands
  associated with $\Lambda$.  The AR-quiver of $\kbp$
  consists of:
  \begin{itemize}
    \item A class of homogeneous tubes, parametrized by $\mathfrak B
      \times k \times \mathbb Z$.
    \item A class of s-components. If $s_2 = 0$, we get a class of
      tubes of rank $s_1$ parametrized by $\mathbb Z$.  If $s_2 >
      0$, we get $s_2$ components of type $\zainf$ with $\tau^{s} =
      [s_2]$.
    \item  $r_2$ components of type $\zainf$ with $\tau^{r} =
      [r_2]$.
    \item A class of $\zainfinf$-components containing all the stalk
      complexes corresponding to a vertex of type $C$, $D$ or $D'$,
      parametrized by $\mathbb Z$.
    \item A class of $\zainfinf$ components parametrized by $\mathfrak
      C \times \mathbb Z$.
    \end{itemize}
\end{theorem}

\begin{proof}
  By Lemma \ref{reduksjonstyper}, any homotopy string admissibly
  reduces to one of three types.  Hence, by Corollary \ref{allcomp},
  Lemma \ref{lem:centralhomstr} and Lemma \ref{lem:uglycomp}, the list
  in the theorem gives all AR-components.
\end{proof}

\section{Example} \label{eksempler} 

In this section, we revisit Example \ref{ex:hovedeks}. The following
notation will be used: The arrows of the quivers are denoted by small
letters (e.g. $a$, $b$), and for an arrow $a$ the inverse is denoted
by $\overline a$.

Recall that $\Lambda = kQ/I$ is a cluster-tilted algebra of type
$\widetilde{A}_{15}$, where $Q$ is the quiver in Figure \ref{16kogger}
and \[ I =\left\langle ih, gi, hg, ed, fe, df, ba, cb, ac, ts, ut, su,
  qp, rq, pr \right\rangle.\] The parameters of $Q$ are $r_1 = 2$, $r_2
= 3$, $s_1 = 4$ and $s_2 = 2$.  We will now give part of the
AR-structure of $\kbp$.

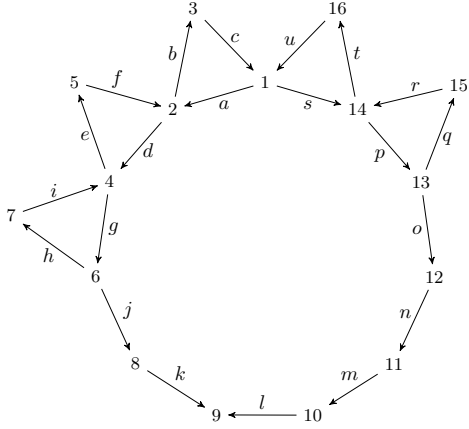
\begin{figure}[ht]
\[
\scalebox{.75}{\begin{tikzpicture}
  \node (1) at (90:3cm) [yvertex] {$1$};
  \node (2) at (122.7:3cm) [yvertex] {$2$};
  \node (4) at (155.4:3cm) [yvertex] {$4$};
  \node (6) at (188.1:3cm) [yvertex] {$6$};
  \node (8) at (220.8:3cm) [yvertex] {$8$};
  \node (9) at (253.5:3cm) [yvertex] {$9$};
  \node (10) at (286.2:3cm) [yvertex] {$10$};
  \node (11) at (318.9:3cm) [yvertex] {$11$};
  \node (12) at (351.6:3cm) [yvertex] {$12$};
  \node (13) at (24.3:3cm) [yvertex] {$13$};
   \node (14) at (57.3:3cm) [yvertex] {$14$};
  
  \node (3) at (106.3:4.5cm) [yvertex] {$3$};
  \node (5) at (138:4.5cm) [yvertex] {$5$};
  \node (7) at (171.7:4.5cm) [yvertex] {$7$};
  \node (16) at (73.6:4.5cm) [yvertex] {$16$};
  \node (15) at (41:4.5cm) [yvertex] {$15$};
  \draw [->] (1)--(2);
  \draw [->] (2)--(3);
  \draw [->] (3)--(1);
  \draw [->] (2)--(4);
  \draw [->] (4)--(5);
  \draw [->] (5)--(2);
  \draw [->] (4)--(6);
  \draw [->] (6)--(7);
  \draw [->] (7)--(4);
  \draw [->] (6)--(8);
  \draw [->] (8)--(9);
  \draw [->] (10)--(9);
  \draw [->] (11)--(10);
  \draw [->] (12)--(11);          
  \draw [->] (13)--(12);
  \draw [->] (13)--(15);
  \draw [->] (15)--(14);
  \draw [->] (14)--(13);
  \draw [->] (14)--(16);
  \draw [->] (16)--(1);
  \draw [->] (1)--(14);
  \node [below] at (104.3:2.9cm) {$a$};
  \node [below] at (135:2.9cm) {$d$};
  \node [right] at (171.7:2.9cm) {$g$};
  \node [right] at (201.4:2.8cm) {$j$};
  \node [above] at (238:2.8cm) {$k$};
  \node [above] at (269:2.9cm) {$l$};
  \node [above] at (302:2.8cm) {$m$};
  \node [left] at (338:2.9cm) {$n$};
  \node [left] at (8:2.9cm) {$o$};
  \node [below] at (44:2.8cm) {$p$};
  \node [below] at (75:2.9cm) {$s$};
  \node [left] at (95:3.8cm) {$c$};
  \node [left] at (112:3.8cm) {$b$};
  \node [above] at (133:3.8cm) {$f$};
  \node [left] at (145:3.6cm) {$e$};
  \node [above] at (167:3.8cm) {$i$};
  \node [below] at (177:3.8cm) {$h$};
  \node [right] at (34:3.6cm) {$q$};
  \node [above] at (46:3.8cm) {$r$};
  \node [right] at (68:3.8cm) {$t$};
  \node [left] at (80:3.8cm) {$u$};
\end{tikzpicture}}
\]
\caption{The quiver $Q$.}\label{16kogger}
\end{figure}

Recall from Section \ref{sec:bastian} that the 16 vertices can be
divided into disjoint sets as follows: $A = \{3,5,7\}$, $A' =
\{15,16\}$, $B = \{2,4\}$, $B' = \{14\}$, $C = \{9\}$, $D = \{6,8\}$,
$D' = \{10,11,12,13\}$ and $F = \{1\}$.

The steps of the clockwise r-walk starting in vertex $7$ are \[ [ei, bf,
\overline{k}lmnopsc, \overline{j}, h, ei, bf, \ldots]. \] The steps of
the clockwise r-walk starting in vertex $5$ consists of the same
steps as for $7$, but deleting the first step.

By Theorem \ref{finalthm}, we know that there are two classes of
characteristic components containing string complexes, or more
precisely, one class of s-components and one of
r-components. Moreover, since neither
of $r_2$ and $s_2$ are zero, we know that both classes consists of
$\zainf$-components. For an r-component, the edge is given by
Proposition \ref{randproposisjon}. We look at the component including
the stalk complex $1_7[0]$.

For any complex on the edge of this component, the upper right
diagonal of the complex is given by Corollary
\ref{allcharacteristicstrings}. In particular, the upper right
diagonal of $1_7[0]$ is $(1_7[0], ei[-1], bfei[-2], \ldots)$.  Figure
\ref{komponent1} shows the three lower rows of the AR-component. Note
that this is the same component as in Figure \ref{edge1}.

\begin{figure}
\[
\resizebox{\linewidth}{!}{
\begin{tikzpicture}[scale=1.3]
  \node (start) at (0,0) [tvertex] {$1_7[0]$};
  \node (ei) at (1,1) [tvertex] {$ei[-1]$};
  \node (rand2) at (2,0) [tvertex] {$1_5[-1]$};
  \node (bf) at (3,1) [tvertex] {$bf[-2]$};
  \node (rand3) at (4,0) [tvertex] {$1_3[-2]$};
  \node (Klmnopsc) at (5,1) [tvertex] {$\overline{k}lmnopsc[-2]$};
  \node (K) at (6,0) [tvertex] {$\overline{k}[-2]$};
  \node (JK) at (7,1) [tvertex] {$\overline{j}\overline{k}[-2]$};
  \node (J) at (8,0) [tvertex] {$\overline{j}[-2]$};
  \node (hJ) at (9,1) [tvertex] {$h\overline{j}[-3]$};
  \node (slutt) at (10,0) [tvertex] {$1_7[-3]$};
  \node (eihJ2) at (0,2) [tvertex] {$eih\overline{J}[-1]$};
  \node (bfei) at (2,2) [tvertex] {$bfei[-2]$};
  \node (Klmnopscbf) at (4,2) [tvertex] {$\overline{k}lmnopscbf[-2]$};
  \node (JKlmnopsc) at (6,2) [tvertex] {$\overline{j} \overline{k}
    lmnopsc[-2]$};
  \node (hJK) at (8,2) [tvertex] {$h\overline{j}\overline{k}[-3]$};
  \node (eihJ) at (10,2) [tvertex] {$eih\overline{j}[-4]$};
  \draw [->] (-0.6,0.6)--(start);
  \draw [->] (start)--(ei);
  \draw [->] (ei)--(rand2);
  \draw [->] (rand2)--(bf);
  \draw [->] (bf)--(rand3);
  \draw [->] (rand3)--(Klmnopsc);
  \draw [->] (Klmnopsc)--(K);
  \draw [->] (K)--(JK);
  \draw [->] (JK)--(J);
  \draw [->] (J)--(hJ);
  \draw [->] (hJ)--(slutt);
  \draw [->] (slutt)--(10.6,0.6);
  \draw [->] (-0.6,1.4)--(eihJ2);
  \draw [->] (eihJ2)--(ei);
  \draw [->] (ei)--(bfei);
  \draw [->] (bfei)--(bf);
  \draw [->] (bf)--(Klmnopscbf);
  \draw [->] (Klmnopscbf)--(Klmnopsc);
  \draw [->] (Klmnopsc)--(JKlmnopsc);
  \draw [->] (JKlmnopsc)--(JK);
  \draw [->] (JK)--(hJK);
  \draw [->] (hJK)--(hJ);
  \draw [->] (hJ)--(eihJ);
  \draw [->] (eihJ)--(10.6,1.4);
  \draw [->] (-0.6,2.6)--(eihJ2);
  \draw [->] (eihJ2)--(0.9,2.9);
  \draw [->] (1.1,2.9)--(bfei);
  \draw [->] (bfei)--(2.9,2.9);
  \draw [->] (3.1,2.9)--(Klmnopscbf);
  \draw [->] (Klmnopscbf)--(4.9,2.9);
  \draw [->] (5.1,2.9)--(JKlmnopsc);
  \draw [->] (JKlmnopsc)--(6.9,2.9);
  \draw [->] (7.1,2.9)--(hJK);
  \draw [->] (hJK)--(8.9,2.9);
  \draw [->] (9.1,2.9)--(eihJ);
  \draw [->] (eihJ)--(10.6,2.6);
  \node [above] at (1,3) [tvertex] {$\vdots$};
  \node [above] at (3,3) [tvertex] {$\vdots$};
  \node [above] at (5,3) [tvertex] {$\vdots$};
  \node [above] at (7,3) [tvertex] {$\vdots$};
  \node [above] at (9,3) [tvertex] {$\vdots$};
  \node [right] at (10.6,1) [tvertex] {$\cdots$};
  \node [left] at (-0.6,1) [tvertex] {$\cdots$};
\end{tikzpicture}}
\] \caption{The lower rows of the AR-component containing $1_7[0]$.}\label{komponent1}
\end{figure}
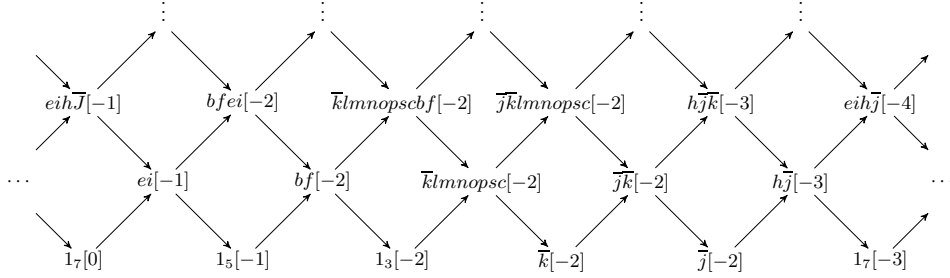

Next, we give the $\zainfinf$-components, up to shift, including the
trivial homotopy strings corresponding to vertices of type $C$, $D$,
and $D'$.  This is the special $\zainfinf$-component described in
Lemma \ref{lem:uglycomp}.  The component is given in Figure
\ref{uglycomp}.

Further, by Tables \ref{omegaplus}--\ref{omegaminustrivial}, we get
that the upper right diagonals starting in the stalk complexes
$1_8^{-1}[0]$, $1_6^{-1}[0]$ and $1_{13}[0]$ consists of subsequent
steps of a counter-clockwise s-walk; the lower left diagonals
starting in $1_9^{-1}[0]$, $1_8^{-1}[0]$ and $1_6^{-1}[0]$ consists of
subsequent steps of a clockwise s-walk; the lower right diagonals
starting in $1_6^{-1}[0]$, $1_{10}[0]$, $1_{11}[0]$, $1_{12}[0]$ and
$1_{13}[0]$ are inverse steps of a clockwise r-walk; and finally
that the upper left diagonals starting in $1_9^{-1}[0]$, $1_{10}[0]$,
$1_{11}[0]$, $1_{12}[0]$ and $1_{13}[0]$ are inverse steps of a
counter-clockwise r-walk.  Using this, we have a complete
description of the AR-component.  In Figure \ref{spesialkomp}, a
part containing the stalk complexes is shown.

Some examples of homotopy bands associated with $\Lambda$ are
$\overline{s}\overline{t}\overline{u}\overline{s}\overline{p}\overline{o}\overline{n}\overline{m}\overline{l}kjgdacba$
and $\overline{s}\overline{t}\overline{u}cba$.

\begin{figure}[hbt]
\[
\resizebox{\linewidth}{!}{
\begin{tikzpicture}[scale=1.3]
  \node (9) at (0,0) [tvertex] {$1_9^{-1}[0]$};
  \node (10) at (1,1) [tvertex] {$1_{10}[0]$};
  \node (11) at (2,2) [tvertex] {$1_{11}[0]$};
  \node (12) at (3,3) [tvertex] {$1_{12}[0]$};
  \node (13) at (4,4) [tvertex] {$1_{13}[0]$};
  \node (8) at (1,-1) [tvertex] {$1_8^{-1}[0]$};
  \node (6) at (2,-2) [tvertex] {$1_6^{-1}[0]$};
  \draw [->] (9)--(10);
  \draw [->] (10)--(11);
  \draw [->] (11)--(12);
  \draw [->] (12)--(13);
  \draw [->] (9)--(8);
  \draw [->] (8)--(6);
  \node (Lk) at (2,0) [tvertex] {$\overline l k[0]$};
  \node (MLk) at (3,1) [tvertex] {$\overline{l}\overline{m} k [0]$};
  \node (NMLk) at (4,2) [tvertex]
  {$\overline{n}\overline{m}\overline{l}k[0]$};
  \node (ONMLk) at (5,3) [tvertex]
  {$\overline{o}\overline{n}\overline{m}\overline{l}k[0]$};
  \node (Lkj) at (3,-1) [tvertex] {$\overline{l}kj[0]$};
  \node (MLkj) at (4,0) [tvertex] {$\overline{m}\overline{l}kj[0]$};
  \node (NMLkj) at (5,1) [tvertex]
  {$\overline{n}\overline{m}\overline{l}kj[0]$};
  \node (ONMLkj) at (6,2) [tvertex]
  {$\overline{o}\overline{n}\overline{m}\overline{l}kj[0]$};
  \draw [->] (10)--(Lk);
  \draw [->] (11)--(MLk);
  \draw [->] (12)--(NMLk);
  \draw [->] (13)--(ONMLk);
  \draw [->] (Lk)--(Lkj);
  \draw [->] (MLk)--(MLkj);
  \draw [->] (NMLk)--(NMLkj);
  \draw [->] (ONMLk)--(ONMLkj);
  \draw [->] (8)--(Lk);
  \draw [->] (Lk)--(MLk);
  \draw [->] (MLk)--(NMLk);
  \draw [->] (NMLk)--(ONMLk);
  \draw [->] (6)--(Lkj);
  \draw [->] (Lkj)--(MLkj);
  \draw [->] (MLkj)--(NMLkj);
  \draw [->] (NMLkj)--(ONMLkj);
  \node (lmnopsc) at (-1,1) [tvertex] {$lmnopsc[0]$};
  \node (mnopsc) at (0,2) [tvertex] {$mnopsc[0]$};
  \node (nopsc) at (1,3) [tvertex] {$nopsc[0]$};
  \node (opsc) at (2,4) [tvertex] {$opsc[0]$};
  \node (psc) at (3,5) [tvertex] {$psc[0]$};
  \draw [->] (lmnopsc)--(9);
  \draw [->] (mnopsc)--(10);
  \draw [->] (nopsc)--(11);
  \draw [->] (opsc)--(12);
  \draw [->] (psc)--(13);
  \draw [->] (lmnopsc)--(mnopsc);
  \draw [->] (mnopsc)--(nopsc);
  \draw [->] (nopsc)--(opsc);
  \draw [->] (opsc)--(psc);
  \node (H) at (3,-3) [tvertex] {$\overline h[0]$};
  \node (LkjH) at (4,-2) [tvertex] {$\overline{l}kj\overline{h}[0]$};
  \node (MLkjH) at (5,-1) [tvertex] {$\overline{m}\overline{l}kj\overline{h}[0]$};
  \node (NMLkjH) at (6,0) [tvertex] {$\overline{n}\overline{m}\overline{l}kj\overline{h}[0]$};
  \node (ONMLkjH) at (7,1) [tvertex]
  {$\overline{o}\overline{n}\overline{m}\overline{l}kj\overline{h}[0]$};
  \draw [->] (6)--(H);
  \draw [->] (Lkj)--(LkjH);
  \draw [->] (MLkj)--(MLkjH);
  \draw [->] (NMLkj)--(NMLkjH);
  \draw [->] (ONMLkj)--(ONMLkjH);
  \draw [->] (H)--(LkjH);
  \draw [->] (LkjH)--(MLkjH);
  \draw [->] (MLkjH)--(NMLkjH);
  \draw [->] (NMLkjH)--(ONMLkjH);
  \node (UADGJKlmnopsc) at (-2,0) [tvertex]
  {$\overline{u}\overline{a}\overline{d}\overline{g}\overline{j}\overline{k}lmnopsc[1]$};
  \node (UADGJK) at (-1,-1) [tvertex]
  {$\overline{u}\overline{a}\overline{d}\overline{g}\overline{j}\overline{k}[1]$};
  \node (UADGJ) at (0,-2) [tvertex]
  {$\overline{u}\overline{a}\overline{d}\overline{g}\overline{j}[1]$};
  \node (UADG) at (1,-3) [tvertex]
  {$\overline{u}\overline{a}\overline{d}\overline{g}[1]$};
  \node (UADGH) at (2,-4) [tvertex]
  {$\overline{u}\overline{a}\overline{d}\overline{g}\overline{h}[1]$};
  \draw [->] (UADGJKlmnopsc)--(lmnopsc);
  \draw [->] (UADGJK)--(9);
  \draw [->] (UADGJ)--(8);
  \draw [->] (UADG)--(6);
  \draw [->] (UADGH)--(H);
  \draw [->] (UADGJKlmnopsc)--(UADGJK);
  \draw [->] (UADGJK)--(UADGJ);
  \draw [->] (UADGJ)--(UADG);
  \draw [->] (UADG)--(UADGH);
  \node (qpsc) at (4,6) [tvertex] {$qpsc[-1]$};
  \node (q) at (5,5) [tvertex] {$q[-1]$};
  \node (qONMLk) at (6,4) [tvertex]
  {$q\overline{o}\overline{n}\overline{m}\overline{l}k[-1]$};
  \node (qONMLkj) at (7,3) [tvertex]
  {$q\overline{o}\overline{n}\overline{m}\overline{l}kj[-1]$};
  \draw [->] (psc)--(qpsc);
  \draw [->] (13)--(q);
  \draw [->] (ONMLk)--(qONMLk);
  \draw [->] (ONMLkj)--(qONMLkj);
  \draw [->] (ONMLkjH)--(8,2);
  \draw [->] (qpsc)--(q);
  \draw [->] (q)--(qONMLk);
  \draw [->] (qONMLk)--(qONMLkj);
  \draw [->] (qONMLkj)--(8,2);
  \node (trq) at (6,6) [tvertex] {$trq[-2]$};
  \draw [->] (q)--(trq);
  \node (HIE) at (4,-4) [tvertex]
  {$\overline{h}\overline{i}\overline{e}[0]$};
  \draw [->] (H)--(HIE);
  \node (LkjHIE) at (5,-3) [tvertex]
  {$\overline{l}kj\overline{h}\overline{i}\overline{e}[0]$};
  \draw [->] (LkjH)--(LkjHIE);
  \draw [->] (HIE)--(LkjHIE);
  \node (mnopscbf) at (-1,3) [tvertex] {$mnopscbf[0]$};
  \draw [->] (mnopscbf)--(mnopsc);
  \draw [->] (-1.9,1.9)--(lmnopsc);
  \draw [->] (-1.9,2.1)--(mnopscbf);
  \draw [->] (-2.9,0.9)--(UADGJKlmnopsc);
  \draw [->] (0.1,3.9)--(nopsc);
  \draw [->] (mnopscbf)--(-0.1,3.9);
  \draw [->] (1.1,4.9)--(opsc);
  \draw [->] (2.1,5.9)--(psc);
  \draw [->] (3.1,6.9)--(qpsc);
  \draw [->] (qpsc)--(4.9,6.9);
  \draw [->] (5.1,6.9)--(trq);
  \draw [->] (trq)--(6.9,5.1);
  \draw [->] (qONMLk)--(6.9,4.9);
  \draw [->] (qONMLkj)--(7.9,3.9);
  \draw [->] (ONMLkjH)--(7.9,0.1);
  \draw [->] (NMLkjH)--(6.9,-0.9);
  \draw [->] (MLkjH)--(5.9,-1.9);
  \draw [->] (LkjHIE)--(5.9,-2.1);
  \draw [->] (3.1,-4.9)--(HIE);
  \draw [->] (UADGH)--(2.9,-4.9);
  \draw [->] (1.1,-4.9)--(UADGH);
  \draw [->] (0.1,-3.9)--(UADG);
  \draw [->] (-0.9,-2.9)--(UADGJ);
  \draw [->] (-1.9,-1.9)--(UADGJK);
  \draw [->] (-2.9,-0.9)--(UADGJKlmnopsc);
  \draw [->] (-1.9,3.9)--(mnopscbf);
  \draw [->] (HIE)--(4.9,-4.9);
  \draw [->] (LkjHIE)--(5.9,-3.9);
  \draw [->] (trq)--(6.9,6.9);
\end{tikzpicture}}
\] \caption{The special \zainfinf-component of $\kbp$.}\label{spesialkomp}\end{figure}

\appendix
\section{Assignments of $S$ and $T$}\label{SandT}
Recall Definition \ref{gentlealternativ} of a gentle algebra.  We fix
the functions $S,T:Q_1 \rightarrow \{-1,1\}$ for a quiver $Q$ in
$\mathcal Q_n^{\star}$. For the arrows $\alpha_i$, where $1 < i \leq
r$, we set $S\alpha_i = -1$ and $T\alpha_i = 1$. Similarly, for the
arrows $\beta_j$, we set $S\beta_j = -1$ and $T\beta_j = 1$ for $1
\leq j < s$. For any 3-cycle containing neither $\alpha_1$ nor
$\beta_s$, we assign values of $S$ and $T$ as shown in Figure
\ref{standard}, where the arrow marked $a$ is either $\alpha_i$ or
$\beta_j$ for some $1 < i \leq r$ or $1\leq j <s$ (note that
$Sa$ and $Ta$ are already taken care of by the above assignments). For
the arrow $\alpha_1$, we set $S\alpha_1 = 1 = T\alpha_1$, and for the
arrow $\beta_s$, we set $S\beta_s = -1 = T\beta_s$. In the case where
$\alpha_1$, respectively $\beta_s$, is part of a 3-cycle, the values
of $S$ and $T$ for the remaining arrows of the 3-cycle are shown in
Figures \ref{alpha1} and \ref{betas}, respectively.

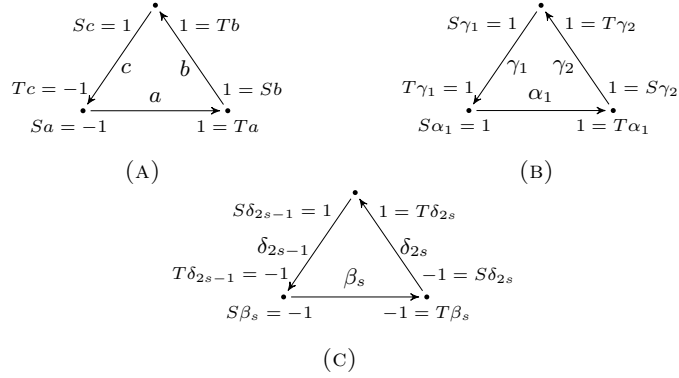
\begin{figure}[h]
  \centering
  \begin{subfigure}[b]{0.4\textwidth}
    \centering
    \resizebox{!}{1.85cm}{
    \begin{tikzpicture}
      \node (1) at (-1,0) [nvertex] {};
      \node (2) at (1,0) [nvertex] {};
      \node (3) at (0,1.47) [nvertex] {};
      \draw [->] (1)--(2);
      \draw [->] (2)--(3);
      \draw [->] (3)--(1);
      \node[above] at (0,0) {\small $a$};
      \node[below] at (-1.2,0) {\scriptsize $Sa=-1$};
      \node[below] at (1,0) {\scriptsize $1=Ta$};
      \node[left] at (0.6,0.6) {\small $b$};
      \node[right] at (0.8,0.3) {\scriptsize $1=Sb$};
      \node[right] at (0.2,1.2) {\scriptsize $1=Tb$};
      \node[right] at (-0.6,0.6) {\small $c$};
      \node[left] at (-0.2,1.2) {\scriptsize $Sc = 1$};
      \node[left] at (-0.8,0.3) {\scriptsize $Tc = -1$};
\end{tikzpicture}}
    \caption{}
    \label{standard}
  \end{subfigure}
  \begin{subfigure}[b]{0.4\textwidth}
    \centering
    \resizebox{!}{1.85cm}{
    \begin{tikzpicture}
      \node (1) at (-1,0) [nvertex] {};
      \node (2) at (1,0) [nvertex] {};
      \node (3) at (0,1.47) [nvertex] {};
      \draw [->] (1)--(2);
      \draw [->] (2)--(3);
      \draw [->] (3)--(1);
      \node[above] at (0,0) {\small $\alpha_1$};
      \node[below] at (-1.2,0) {\scriptsize $S\alpha_1=1$};
      \node[below] at (1,0) {\scriptsize $1=T\alpha_1$};
      \node[left] at (0.6,0.6) {\small $\gamma_2$};
      \node[right] at (0.8,0.3) {\scriptsize $1=S\gamma_2$};
      \node[right] at (0.2,1.2) {\scriptsize $1=T\gamma_2$};
      \node[right] at (-0.6,0.6) {\small $\gamma_1$};
      \node[left] at (-0.2,1.2) {\scriptsize $S\gamma_1 = 1$};
      \node[left] at (-0.8,0.3) {\scriptsize $T\gamma_1 = 1$};
\end{tikzpicture}}
    \caption{}
    \label{alpha1}
  \end{subfigure}
  \begin{subfigure}[b]{0.4\textwidth}
    \centering
    \resizebox{!}{1.85cm}{
    \begin{tikzpicture}
      \node (1) at (-1,0) [nvertex] {};
      \node (2) at (1,0) [nvertex] {};
      \node (3) at (0,1.47) [nvertex] {};
      \draw [->] (1)--(2);
      \draw [->] (2)--(3);
      \draw [->] (3)--(1);
      \node[above] at (0,0) {\small $\beta_s$};
      \node[below] at (-1.2,0) {\scriptsize $S\beta_s=-1$};
      \node[below] at (1,0) {\scriptsize $-1=T\beta_s$};
      \node[right] at (0.5,0.7) {\small $\delta_{2s}$};
      \node[right] at (0.8,0.3) {\scriptsize $-1=S\delta_{2s}$};
      \node[right] at (0.2,1.2) {\scriptsize $1=T\delta_{2s}$};
      \node[left] at (-0.5,0.7) {\small $\delta_{2s-1}$};
      \node[left] at (-0.2,1.2) {\scriptsize $S\delta_{2s-1} = 1$};
      \node[left] at (-0.8,0.3) {\scriptsize $T\delta_{2s-1} = -1$};
\end{tikzpicture}}
    \caption{}
    \label{betas}
  \end{subfigure}
  \caption{Assignment of $S$ and $T$ to the arrows of a
    3-cycle.}\label{ST}
\end{figure}

\begin{lemma}\label{stringfunctions}
If $\Lambda \cong kQ/I$ is a cluster-tilted algebra of type
$\ant$, then the above assignments of $S$ and $T$ to the quiver $Q$ in $\mathcal
Q_n^{\star}$ satisfies the conditions given in Definition \ref{gentlealternativ}.
\end{lemma}

The proof is straight-forward. Note that this assignment of $S$ and
$T$ is not unique. A different assignment of the functions will still
give the same final result, but to give examples and technical
results in an unequivocally manner, we need to fix explicitly given
functions.

\section{The longest antipath $\theta_{x,\varepsilon}$}
An \emph{antipath} in a gentle algebra is either a trivial homotopy
string, or a direct homotopy string $\rho$ such that for any two
consecutive arrows $\alpha$ and $\beta$ in $\rho$, we have that their
composition is a relation.  Bobi\'nski defines, for each vertex $x \in
Q_0$ and each $\varepsilon \in \{-1,1\}$, the set
$\Theta_{x,\varepsilon}$ consisting of all antipaths $\theta$ such
that $t\theta = x$ and $T\theta = \varepsilon$.  Moreover, he defines
$\theta_{x,\varepsilon}$ to be the maximal element of
$\Theta_{x,\varepsilon}$ if such element exists; otherwise, he defines
$\theta_{x,\varepsilon} = \emptyset$.

We will now consider the possible values of $\theta_{x,\varepsilon}$
when $Q$ is a quiver in $\mathcal Q_n^{\star}$.  The cases where
$\theta_{x,\varepsilon}$ is not the empty string are the following:
\begin{itemize}
\item $\theta_{A_i,-1} = 1_{A_i}$ for $1 \leq i \leq r_2$,
\item $\theta_{C,1}=\alpha_r$ when $r_1 > 0$,
\item $\theta_{C,-1}=\beta_s$ when $s_1 > 0$,
\item $\theta_{D_i,1}=\alpha_{i+r_2}$ for $1 \leq i \leq r_1-1$,
\item $\theta_{D_i,-1}=1_{D_i}$ for $0 \leq i \leq r_1-1$ and $r_2 >
  0$,
\item $\theta_{D_i,-1}=1_{D_i}$ for $1 \leq i \leq r_1-1$ and $r_2 = 0$,
\item $\theta_{D'_j,1}=\beta_{j+s_2}$ for $1 \leq j \leq s_1-1$,
\item $\theta_{D'_j,-1}=1_{D'_i}$ for $0 \leq j \leq s_1-1$.
\end{itemize}
For the rest of the $\theta_{x,\varepsilon}$'s, the set
$\Theta_{x,\varepsilon}$ is infinite with no maximal element. This is
the case when there exists an arrow $\alpha$ which is part of a
3-cycle and such that $t\alpha = x$ and $T\alpha = \varepsilon$.

\section{Proof of Proposition
  \ref{tabellproposisjon}} \label{bobinskiappendix} 

We now give a proof of Proposition \ref{tabellproposisjon} in section
\ref{almostsplittriangels}. In this proof we will use the algorithm
presented by Bobi\'{n}ski in ~\cite[Section 6]{bob}, which we will now
recall.  Let $\omega=\alpha_l \cdots \alpha_1$ be a homotopy string
with homotopy partition $\omega = \sigma_L \cdots \sigma_1$.  Note
that our convention of labeling the letters and homotopy letters is
opposite of the convention used in \cite{bob}.

We now state Bobi\'nski's algorithm for finding $\omega^+$. If
$\alpha_l$ is a direct letter, let $\rho(\omega)$ denote the maximal
integer $i \in [1,l]$ such that the $i$ last letters of $\omega$, that
is $\alpha_l \cdots \alpha_{l-i+1}$, is an antipath.  If $\alpha_l$ is
an inverse letter or $\omega$ is a trivial homotopy string, we let
$\rho(\omega) = 0$.  Next, we define the substring $\omega'$ of
$\omega$ to be
\[
\omega' = 
\begin{cases}
  \omega & \text{if $\rho(\omega)= 0$} \\
  1_{s\omega}^{S\omega} & \text{if $\rho(\omega) = l$}\\
  \alpha_{l-\rho(\omega)}\cdots \alpha_1 & \text{if $0<\rho(\omega)<l$}\\
\end{cases}
\]
Note that $\omega'$ is obtained by removing the longest possible
antipath from the end of $\omega$, and that $\rho(\omega)$ denotes the
number of letters removed.

Let $\sigma$ denote the maximal path of $Q$ with no subpath in $I$
such that $\sigma\cdot \omega$ is defined as composition of homotopy
strings.  Now, we have 6 cases for computing $\omega^+$, as listed
below:

\[
\omega^+ = \left\{ 
{\renewcommand{\arraystretch}{1.4}%
\begin{array}{lll}
\theta^{-1}_{t\sigma,-T\sigma}\sigma \omega & l(\sigma) > 0 & (1) \\
\theta^{-1}_{t\omega',-T\omega'}\omega' & l(\sigma) = 0\text{, }~ l(\theta_{t\omega',
  -T\omega'}) > 0\text{, and }~ l(\omega') > 0 & (2) \\
({}^{[1]}\theta_{t\omega',-T\omega'})^{-1} & l(\sigma) = 0\text{, }~ l(\theta_{t\omega',
  -T\omega'}) > 0\text{, and }~ l(\omega') = 0 &  (3) \\
{}^{[1]}\omega' & l(\sigma) = 0\text{, }~ l(\theta_{t\omega',
  -T\omega'}) = 0{, }~ l(\omega') > 0\text{,} & \\ 
& \text{and } \alpha_{l(\omega')}(\omega') \text{ is an inverse arrow} & (4) \\
\omega' & l(\sigma) = 0\text{, }~ l(\theta_{t\omega',
  -T\omega'}) = 0{, }~ l(\omega') > 0\text{,} & \\
& \text{and } \alpha_{l(\omega')}(\omega') \text{ is an arrow} & (5)\\
\emptyset & l(\sigma) = 0\text{, }~ l(\theta_{t\omega',
  -T\omega'}) = 0\text{, and }~ l(\omega') = 0 & (6)
\end{array}}\right.
\]

Finally, for an integer $m$, we compute
\[
m'(\omega) = \left\{ \begin{array}{ll}
    l(\theta_{t\sigma,-T\sigma}) - 1 & l(\sigma) > 0 \\
    l(\theta_{t\omega',-T\omega'}) + \rho(\omega) -1 & l(\sigma) = 0 
    \end{array}\right.
\]
and this completely describes how we find the pair
$(m'(\omega),\omega^+)$ from the pair $(m, \omega)$.  This concludes
Bobi\'nski's algorithm. We will now state again Proposition
\ref{tabellproposisjon} and then use the above algorithm to give the
proof.

\begin{proposition}Let $\Lambda$ be a cluster-tilted algebra of type
  $\ant$, and let $\omega[m]$ be a string complex in $\kbp$. The
  middle term $\omega^+[m+m'(\omega)]$ in the AR-triangle starting in
  $\omega$ is given by the entries in Tables \ref{omegaplus} and
  \ref{omegaplustrivial}. The middle term $\omega_-[m-m'(\omega_-)]$
  in the AR-triangle ending in $\omega$ is given by the entries in
  Tables \ref{omegaminus} and \ref{omegaminustrivial}.
\end{proposition}

\begin{proof}
  The proof of this proposition can be done by direct calculations, and
  thus we only include part of the calculations.  We include below the
  calculations for $\omega^+$ when $\alpha_l(\omega)=\alpha_i^{-1}$
  for $1\leq i \leq r$. There are then the following cases to
  consider:
\begin{itemize}
\item If $2\leq i\leq r_2+1$, then $\sigma=\gamma_{2i-2}$ so we are in
  the first case of the algorithm. Now
  $\theta_{t\sigma,-T\sigma}=\theta_{A_{i-1,-1}}=1_{A_i}$, thus
  $\omega^+=\gamma_{2i-2}\omega=cw\_r(t\omega)\omega$ and
  $m'(\omega)=-1$.
\item If $r_2+2\leq i\leq r$, then it is clear that $\sigma=\emptyset$
  and  $\rho(\omega)=0$ so $\omega'=\omega$. Now
  $\theta_{t\omega',-T\omega'}=\theta_{D_{i-r_2-1},1}=\alpha_{i-1}$ so
  we are in case $(2)$ of the algorithm, and
  $\omega^+=\alpha_{r_2+j-1}^{-1}\omega=cw\_r(t\omega)\omega$ and
  $m'(\omega)=0$.
\item If $i=1$ and $r_2 = 0$, then $\alpha_l(\omega) = \alpha_1^{-1}$
  and $\sigma = \beta_s \cdots \beta_1$.  Hence we are in case $(1)$
  of the algorithm, and $\theta_{t\sigma, -T\sigma} = \alpha_r$.  Then
  $\omega^+ = \alpha_r^{-1}\beta_s \cdots \beta_1 \omega =
  cw\_r(t\omega)\omega$ and $m'(\omega)=0$.
\item If $i=1$ and $r_2 > 0$, then $\alpha_l(\omega) = \alpha_1^{-1}$
  and
  \[
  \sigma = \begin{cases}
    \beta_s \cdots \beta_1 & r_1 > 0 \\
    \gamma_{2r_2} \beta_s \cdots \beta_1 & r_1 = 0 ~.
  \end{cases}
  \]
  In both cases, we are in case $(1)$ of the algorithm, and
  \[
  \theta_{t\sigma, -T\sigma} = \begin{cases}
    \alpha_r & r_1 > 0 \\
    1_{A_{r_2}} & r_1 = 0 ~.
  \end{cases}
  \]
  In both cases, we get $\omega^+ = cw\_r(t\omega)\omega$, with
  $m'(\omega) = 0$ if $r_1 > 0$ and $m'(\omega) = -1$ if $r_1 = 0$.
\end{itemize} 
The calculations for Tables \ref{omegaplus} and \ref{omegaplustrivial}
can be done similarly.

For the proof of Table \ref{omegaminus} and \ref{omegaminustrivial},
the following procedure is applied: For every entry of the table, let
$\omega$ be a homotopy string satisfying the given description of the
entry and let $\widehat{\omega}$ be the homotopy string arising by
performing the operation from the corresponding entry in the third
column (i.e. $\widehat{\omega} = \omega_-$).  Next, consider all
possible $\widetilde{\omega}$ such that $\widetilde{\omega}^+ =
\omega$ (we use Tables \ref{omegaplus} and \ref{omegaplustrivial} for
this).  Then, for each such $\widetilde{\omega}$, verify that
$\widetilde{\omega}=\widehat{\omega}$.  In this proof, we shall only
do this for the entry $\alpha_l(\omega) = \gamma_{2i}$ in Table
\ref{omegaminus}. The rest is left to the reader.

Assume that $\omega$ is a homotopy string with $\alpha_l(\omega) =
\gamma_{2i}$ for some $1 \leq i \leq r_2$. We examine Tables
\ref{omegaplus} and \ref{omegaplustrivial} to find
$\widetilde{\omega}$ such that $\widetilde{\omega}^+$ can be a
homotopy string ending with such an arrow.  The possibilities for
$\widetilde{\omega}$ are:
\begin{itemize}
\item $\alpha_l(\widetilde{\omega}) = \alpha_i$ for $1 \leq i \leq
  r_2$
\item $\alpha_l(\widetilde{\omega}) = \alpha_i^{-1}$ for $2 \leq i
  \leq r_2 + 1$
\item $\alpha_l(\widetilde{\omega}) = \gamma_{2i} $ for $2 \leq i \leq
  r_2$ or $i = 1$ and $r_1 = 0$
\item $\alpha_l(\widetilde{\omega}) = \delta_{2i-1}$ for $1 \leq i
  \leq s_2$ and $r_1 = 0$
\item $\alpha_l(\widetilde{\omega}) = \delta_{2i}^{-1}$ for $1 \leq i
  \leq s_2$ and $r_1 = 0$
\end{itemize}
It is easy to see that those are all homotopy strings such that
$\widetilde{\omega}^{+}$ is
$cw\_r(t\widetilde{\omega})\widetilde{\omega}$.  We have that
$\widehat{\omega}$ is the clockwise r-reduction of $\omega$, an since
this operation undoes adding a step of a clockwise r-walk, we will
always get $\widetilde{\omega} = \widehat{\omega}$.
\end{proof}

\bibliographystyle{plain}

\end{document}